\documentclass[article,ij4uq]{ij4uq}              

\usepackage[hang]{footmisc}
\usepackage{etex}
\usepackage{amsfonts}
\usepackage{amssymb, amsmath}
\usepackage{array}
\usepackage{ragged2e}
\usepackage[english]{babel} 
\usepackage{ifthen}
\usepackage{amsthm} 	
\usepackage{bbm}		
\usepackage{graphicx}
\usepackage{epstopdf}
\usepackage{lmodern}
\usepackage{algorithmic}
\usepackage{wrapfig} 

\usepackage{cancel}
\usepackage{tikz}	
\usepackage{mathtools} 
\usepackage{ragged2e}
\usepackage{tabularx}
\usepackage{nicefrac}
\usepackage{multirow}
\usepackage[colorlinks=true]{hyperref}
\hypersetup{urlcolor=blue, citecolor=blue}
\usepackage{lineno}
\usepackage{tabto}
\usepackage{yhmath}

\newcommand{\indikator}{\mathbb{I}}
\def \zeros{\mathbb{O}}

\def \d{\, \textup{d}}
\def \T{\textup{T}}
\def \GK{\Pi_\mathbb{K}}
\def \diag{\textup{diag}}

\def \C{\mathbb{C}}
\def \NKL{M}

\def \R{\mathcal{R}}

\def \yI{v}
\def \yII{\sigma}
\def \Uniform{\mathcal{U}}
\def \Normal{\mathcal{N}}

\def \Dhat{\boldsymbol{\widehat{\mathcal{D}}}}
\def \DhatPlus{\boldsymbol{\widehat{\mathcal{D}}^+}}
\def \DhatMinus{\boldsymbol{\widehat{\mathcal{D}}^-}}

\def \Lambdasteady{\boldsymbol{\widehat{\mathcal{D}}}}
\def \DerivStress{\bar{\epsilon}^p_\sigma(\sigma^*)}

\DeclareMathOperator*{\argmin}{argmin}

\def \Ahat{\boldsymbol{\widehat{A}}}

\def \SRhat{\boldsymbol{\widehat{C}}}
\def \Shat{\boldsymbol{\widehat{S}}}
\def \Bhat{\boldsymbol{\widehat{\mathcal{Q}}}}
\def \Ghat{\boldsymbol{\widehat{\mathcal{B}}}}

\def \Tsteady{T}
\def \Lambdasteady{\Lambda}
\def \lambdasteady{\lambda}
\def \Riemann{\mathcal{R}}
\def \SRiemann{C}
\def \GRiemann{B}
\def \Gphysicslinear{B_y}
\def \Gphysicsnonlinear{\mathcal{B}_{\widehat{y}}}

\def \langleP{\langle} 

\def \lVertP{\lVert} 

\def \K{\mathbb{K}}
\def \Rhat{\boldsymbol{\widehat{\mathcal{R}}}}
\def \Rd{\boldsymbol{\widehat{\zeta}}}
\def \ZetaPlus{\boldsymbol{\widehat{\zeta}^+}}
\def \ZetaMinus{\boldsymbol{\widehat{\zeta}^-}}

\def \mub{\hat{\mu}}
\def \smallestEV{\sigma_{\min}}
\def \largestEV{\sigma_{\max}}
\def \S{\mathcal{S}}
\definecolor{gruen}{rgb}{0,0.5,0}
\definecolor{braun}{rgb}{0.5,0.25,0.25}
\def \process{\textit{Pr}}
\def \M{\mathcal{M}}
\def \H{\mathcal{H}}
\def \L{\mathcal{L}}
\def \D{\textup{D}}

\def \lambdaMatern{\lambda_s}
\def \lambdaKL{d}

\usetikzlibrary{arrows,shapes,decorations.pathreplacing,decorations.markings}
\definecolor{lila}{rgb}{0.9,0,1}
\definecolor{gruen}{rgb}{0,0.5,0}
\definecolor{braun}{rgb}{0.5,0.25,0.25}
\definecolor{farbe1}{rgb}{0.75,0.34,0}
\definecolor{myblue}{rgb}{0.196,0.196,0.694}
\definecolor{myred}{rgb}{1,0.3,0.1}

\fancypagestyle{plain}{%
  \fancyhf{}
  \fancyhead[R]{\small {\it } }
  \fancyfoot[R]{\small\bf\thepage }
  \fancyfoot[L]{\fottitle}
  }

\begin{document}

\volume{}
\title{Feedback control for random, linear hyperbolic balance laws}
\titlehead{Feedback control for random, linear hyperbolic balance laws}
\authorhead{Markus Bambach, Stephan Gerster, Michael Herty, Muhammad Imran}

\author[1]{Markus Bambach}
\corrauthor[2]{Stephan Gerster}
\author[2]{Michael Herty}
\author[3]{Muhammad Imran}
\corremail{gerster@igpm.rwth-aachen.de}

\corraddress{Institut f\"{u}r Geometrie und Praktische Mathematik,  
	RWTH Aachen University,  
	Templergraben~55, 52062~Aachen, Germany}

\address[1]{Institut f\"ur virtuelle Produktion,   TH Z\"urich,  
	Technoparkstra{\ss}e~1, 8005 Z\"urich, Switzerland}

\address[2]{Institut f\"{u}r Geometrie und Praktische Mathematik,  
RWTH Aachen University,  
Templergraben~55, 52062~Aachen, Germany}

\address[3]{Lehrstuhl Konstruktion und Fertigung,  
	BTU Cottbus-Senftenberg,  
	Konrad-Wachsmann-Allee~17,  03046~Cottbus, Germany}
	

\dataO{mm/dd/yyyy}
\dataF{mm/dd/yyyy}

\abstract{
We design boundary controls of physical systems that are faced by uncertainties. The system dynamics are 
described by random hyperbolic balance laws.  
The control aims to steer the system to a desired state under  uncertainties. 
We propose a control based on Lyapunov stability analysis of a suitable series expansion of the random dynamics. 
The control damps the impact 
of uncertainties exponentially fast in time. 
The presented approach can be applied to a large class of physical systems and random perturbations, as~e.g.~Gaussian processes.  
We illustrate the boundary control effect on a stochastic viscoplastic material model. 	
}


\maketitle


\section{Introduction}


Boundary stabilization  has been  studied intensively in the past years~\cite{O1}. 
 A well-known approach to prove  exponential stability of a desired state is the construction of suitable Lyapunov functionals and the analysis of so-called \emph{dissipative boundary conditions}.  Exponential decay of a continuous Lyapunov function under \emph{dissipative} boundary conditions has been proven in~\cite{L7,L1,L4,G4,O1}.
Also  explicit decay rates for numerical schemes have
been established~\cite{D1,D2,Gerster2019,Knapp2019,Weldegiyorgis2020}.

Most results are based on the assumption that model parameters and desired states are known exactly. 
However, often there is need to take uncertainties into account. 
For instance, model parameters are uncertain due to noisy  measurements and due to variations in the behaviour of materials. 
Moreover, epistemic uncertainties arise,  when the considered mathematical models do not exactly describe the true physics as e.g.~in constitutive equations for material models.

When the underlying model is not known exactly, but is given by a probability law or by statistical  moments, the deterministic  stabilization concept should be extended to this stochastic case. 
Monte-Carlo methods may be used to apply deterministic stabilization concepts to each realization. For instance, sampling-based methods are used in~\cite{Gotzes2016,Schuster2019} to analyze the existence of optimal solutions for some optimization problems with probabilistic constraints.

In contrast, the underlying tool used in this paper is the representation of stochastic perturbations by a series of orthogonal functions,  known as generalized polynomial chaos~(gPC) expansions~\cite{S1,S2,S3}. 
Expansions of the stochastic input are substituted into the governing equations and they are projected by a Galerkin method to obtain deterministic evolution equations for the coefficients of this expansion. 
This intrusive approach is often applied in uncertainty quantification.  In this direction many  results for kinetic equations are available~\cite{K1,K2,Zanella2019,Yuhua2017,Yuhua2018,Carrillo2019,Zanella2020}. Recently, also  results for  hyperbolic equations have been established~\cite{H0,H2,H3,H4,S5,GersterJCP2019,GersterHertyCicip2020,KuschMaxPrin2017}.

The problem if the random solutions to kinetic equations converge to the deterministic kinetic equilibrium exponentially fast has been analyzed~\cite{Li2016,Yuhua2018,Liu2018,Hui2020} by using hypocoercivity properties~\cite{Herau2006, Herau2017,Villani2009,Dolbeault_etal2009,Dolbeault_etal2015}.
However, results for stabilizing general hyperbolic systems are only partial. Already in the deterministic case solutions to  hyperbolic conservation laws exist in the classical sense only in finite time due to the occurence of shocks~\cite{Dafermos} and 
Lyapunov's indirect method~\cite{Khalil} does not necessarily hold. 
Furthermore, the results are so far restricted  in the sense that the destabilizing effect of the source term  is sufficiently  small~\cite{Bastin2011,O1,Gugat2019}.
Explicit decay rates for the $L^2$-stabilization of steady states have been established in the simplified case of \emph{linear, deterministic} systems~\cite[Prop.~5.2]{O1}, which include Maxwell's,  elasticity and linearized Euler equations. 

This paper extends these results to the stochastic case.  
The stochastic system is reformulated as a sequence of deterministic problems. A weighted $L^2$-norm of this series is used as Lyapunov function. 
We will show that a modification of the dissipativity condition~\cite[Th.~2.3]{L7} yields an exponentially fast decaying Lyapunov function, which in turn makes the mean and variance of stochastic deviations diminish exponentially fast over time.

The presented approach is related to the 
variance reduction by robust design of boundary conditions~\cite{Wahlstein2015}. There, energy estimates are derived to obtain well-posed boundary value problems and boundary conditions that reduce the variance in the system.  
In extension, the Lyapunov stability analysis~\cite{L7,O1} damps deviations~\emph{exponentially fast} over time as long as stronger assumptions on the boundary conditions and the source term are satisfied. 


This paper is structured as follows. 
Section~\ref{Sec1} is devoted to the representation of stochastic processes with orthogonal functions. 	
We derive a deterministic formulation of the underlying stochastic boundary value problem. 
Section~\ref{Sec4} considers a stability analysis for general linear hyperbolic balance laws. A Lyapunov function that gives an upper bound on deviations from desired states is presented. Furthermore, conditions on the boundary control are stated that make the Lyapunov functions and hence the deviations decay exponentially fast over time. 
Possible extensions to nonlinear systems are discussed in Section~\ref{SectionNonlinear}.  
Finally, Section~\ref{SectionFeedback} illustrates the presented approach by steering a viscoplastic material to a desired state under uncertainties. 


\section{Random hyperbolic boundary value problems}\label{Sec1}

Uncertainties are included by a possibly multidimensional random variable~$\xi:\Omega\rightarrow \mathbb{R}^{\NKL}$ that is defined on a probability space~$\big(\Omega,\mathcal{F}(\Omega),\mathbb{P}\big)$. 
The space- and time-depending  dynamics~${y\big(t,x;\xi(\omega)\big) \in\mathbb{R}^2}$ are for each fixed realization~$\omega\in\Omega$ described by ${2\times2}$ strictly hyperbolic balance laws of the form
\begin{equation}\label{LyapunovDev1}
\partial_t y\big(t,x;\xi(\omega)\big)
+
A\big(x;\xi(\omega)\big)
\partial_x y\big(t,x;\xi(\omega)\big)
=
-S\big(x;\xi(\omega)\big)
y\big(t,x;\xi(\omega)\big) \\
\quad
\text{for}
\quad
y =
(
v, \sigma
)^\T. 
\end{equation}
The matrix~${A\big(x;\xi(\omega)\big)\in\mathbb{R}^{2\times 2}}$, which describes advection, is assumed  diagonalizable, Lipschitz-continuous and the source term~${S\big(x;\xi(\omega)\big)\in\mathbb{R}^{2\times 2}}$ is assumed continuous in~${x\in[0,L]}$.
A typical model, which we consider in Section~\ref{SectionFeedback},  is a viscoplastic deformation described by a displacement velocity or acceleration~$v$ and stress~$\sigma$. 
If a linear control~$\Gphysicslinear\in\mathbb{R}^{2\times 2}$ is applied at the boundaries of the spatial  domain~$[0,L]$,  the desired control is of the form 
\begin{equation}\label{BC}
\begin{pmatrix}
{\yI}(t,0;\xi) \\  {\yI}(t,L;\xi)
\end{pmatrix}
=
\Gphysicslinear
\begin{pmatrix}
{\yII}(t,0;\xi) \\  {\yII}(t,L;\xi)
\end{pmatrix}.
\end{equation}


\noindent
Various choices of the control~$\Gphysicslinear$ are possible and a suitable choice will be presented in Section~\ref{SubSecControlRule}. 
In particular, we are interested in the control of deviations from a desired state. We introduce the notation 
$$
\Delta y(t,x;\xi) 
\coloneqq
y(t,x;\xi)-y^*(x;\xi),
$$
 where 
$y^*(x;\xi)$ denotes the desired state, which may include also uncertainties.  
The analysis is described for dynamics on a one-dimensional physFical domain~$[0,L]\subset\mathbb{R}$, i.e.~one arc. Dynamics~$y^{(j)} = (v^{(j)},\sigma^{(j)})^\T$ on multiple arcs $j=1,\ldots,n$ can be coupled by specifying appropriate node  conditions, as illustrated in Figure~\ref{FigNetwork0}. 

\vspace{6mm}
\begin{figure}[h]
	\tikzstyle{block} = [thick,draw, fill=blue!20, rectangle, 
minimum height=3em, minimum width=3em]
\tikzstyle{Kreis} = [thick,draw, fill=blue!20, circle, 
minimum height=3em, minimum width=3em]
\tikzstyle{sum} = [draw, fill=darkgreen!20, rectangle, 
minimum height=2em, minimum width=3em]
\tikzstyle{input} = [coordinate]
\tikzstyle{output} = [coordinate]
\tikzstyle{pinstyle} = [ pin edge={to-,thick,black}]
\tikzstyle{vecArrow} = [thick,double distance=25pt]

\vspace{-1cm}\hspace{-5mm}
\begin{tikzpicture}[scale=0.1]

\draw [vecArrow] (0,0) -- (50,0);
\draw [vecArrow] (155,10) -- (80,2.3);
\draw [vecArrow] (155,-10) -- (80,-2.3);

\node [block,align=center] at (75,0) {\textbf{node conditions:}\\
\	$
\begin{pmatrix} \boldsymbol{v} (t,0;\xi) \\ \boldsymbol{v}(t,L;\xi) \end{pmatrix} 
=
\Gphysicslinear
\begin{pmatrix}  \boldsymbol{\sigma} (t,0;\xi) \\ \boldsymbol{\sigma} (t,L;\xi) \end{pmatrix} 
$ \
  };

\node at (30,0) {$y^{(1)}(x;\xi) + \Delta y^{(1)}(t,x;\xi) $};
\node at (121,6.5) {$y^{(2)}(x;\xi) + \Delta y^{(2)}(t,x;\xi) $};
\node at (121,-6.8) {$y^{(n)}(x;\xi) + \Delta y^{(n)}(t,x;\xi) $};

\node at (4.5,0) {arc 1};
\node at (148,10) {arc 2};
\node at (148,-10) {arc $n$};

\node at (60,12) {
$ \boxed{\boldsymbol{v}\coloneqq \big( v^{(1)},\ldots,v^{(n)} \big)^\T}$
};

\node at (60,-12) {
	$\boxed{ \boldsymbol{\sigma}\coloneqq \big( \sigma^{(1)},\ldots,\sigma^{(n)} \big)^\T}$
};







\end{tikzpicture}
	\caption{Network with $n$ arcs in a perturbed steady state with linear boundary conditions~$\Gphysicslinear\in\mathbb{R}^{2n\times 2n}$.}
	\label{FigNetwork0}
\end{figure}
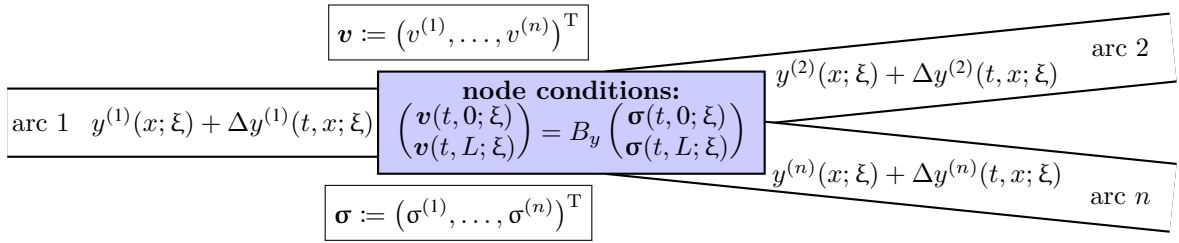

\noindent
A stabilization strategy could be to apply deterministic concepts to each possible realization.  This ``worst-case'' ansatz, however, depends sensitively on possibly large deviation in the realizations. 
A more efficient stabilization concept is a  mean squared error~${
	\mathbb{E} \big[
	\lVert \Delta y(t,x;\xi) 
	\rVert^2 \big]
}$ 
that should decay over time as fast as possible.
Although negligibly unlikely events are permitted,  the stabilization concept is still strong. 
It makes both the mean~$\mathbb{E}$ and the variance~$\mathbb{V}$ of deviations decay, since we have 
\begin{equation}\label{MSE}
\mathbb{E} \bigg[
\big\lVert \Delta y(t,\cdot;\xi) 
\big\rVert^2 \bigg]
=
\mathbb{E} \Big[
\big\lVert \Delta y(t,\cdot;\xi) 
\big\rVert\Big]^2
+
\mathbb{V} \Big[
\big\lVert \Delta y(t,\cdot;\xi) 
\big\rVert\Big]
\end{equation}
for any suitable norm~$\lVert \cdot \rVert$. 
In the following we will consider, unless otherwise stated, the~$L^2$-norm~$\lVert\cdot\rVert_{L_2}$ with respect to space~$x\in[0,L]$, i.e.
$${
\mathbb{E} \bigg[
\big\lVert \Delta y(t,\cdot;\xi) 
\big\rVert^2_{L^2} \bigg]
=
\mathbb{E} \bigg[
\int
\Delta v(t,x;\xi)^2 
+
\Delta \sigma(t,x;\xi)^2 
\d x
\bigg]
\quad\text{with}\quad
\Delta y =
\begin{pmatrix}
\Delta v \\ \Delta \sigma
\end{pmatrix}.
}$$
Furthermore, this stabilization concept allows for an equivalent  representation of the mean squared error~\eqref{MSE} by a series of deterministic equations, which will be introduced in this section.

\subsection{Representation of stochastic processes}\label{Sec2}

We consider a stochastic process~$\process \in L^2\big( (0,L);\mathbb{R} \big) \bigotimes \mathbb{L}^2(\Omega,\mathbb{P}) $ that admits a tensor product structure such that  for each fixed point in space ${x\in(0,L)}$ the process belongs to the $\mathbb{L}^2$-space
\begin{alignat*}{8}
&\mathbb{L}^2(\Omega,\mathbb{P}) &&\coloneqq \Big\{ \process  \, : \, \Omega \rightarrow L^2\big( (0,L);\mathbb{R} \big), \ \omega \mapsto \process(\cdot;\omega)\ \Big| \ \big\lVertP \process(x;\cdot)  \big\rVert_{\mathbb{P}} < \infty \Big\} \\
& \text{for}  
&&
\big\langleP \process(x_1;\cdot),\process(x_2;\cdot) \big\rangle_{\mathbb{P}} \coloneqq  \int\limits \process(x_1;\omega) \process(x_2;\omega) \d \mathbb{P}(\omega) \\
&\text{and}
&&\big\rVert \process(x;\cdot) \big\rVert_{\mathbb{P}} \coloneqq
\sqrt{
	\big\langle \process(x;\cdot),\process(x;\cdot) \big\rangle_{\mathbb{P}}	
}.
\end{alignat*}
This tensor product structure allows to express the \textbf{expected value} and the \textbf{covariance kernel} by
\begin{align}
\mathbb{E}\big[ \process(x_1;\omega) \process(x_2;\omega) \big] 
&\coloneqq 
\big\langleP \process(x_1;\cdot), \process(x_2;\cdot) \big\rangle_{\mathbb{P}}, \tag{$\mathbb{E}$} \\
\C(x_1,x_2)\coloneqq
\textup{Cov}\big[\process(x_1;\omega),\process(x_2;\omega)\big]  
&\coloneqq \mathbb{E}\big[ \process(x_1;\omega) \process(x_2;\omega) \big]-\mathbb{E}\big[\process(x_1;\omega)\big] \mathbb{E}\big[\process(x_2;\omega)\big].
\tag{$\C$}
\label{CovarianceKernel}
\end{align}

\noindent
A centered stochastic process
that is mean square continuous, i.e.~${
\big\lVert \process(x;\cdot)-\process(x^*;\cdot) \big\rVert_{\mathbb{P}}
\rightarrow 0
}$ for 
${x \rightarrow x^*}$, 
admits the  \textbf{Karhunen-Lo\`{e}ve expansion}
\begin{equation}\label{KL}\tag{KL}
\mathcal{K}[\process](x;\omega)
\coloneqq
\sum_{k=1}^{\NKL}
\sqrt{\lambdaKL_k}
\psi_k(x)
\xi_k(\omega) 
\quad \text{with} \quad
\xi_k(\omega) 
\coloneqq
\frac{1}{\sqrt{\lambdaKL_k}}
\int_0^L
\process(x;\omega) \psi_k(x) 
\d x,
\end{equation}
where~$\psi_k$ and~$\lambdaKL_k>0$ are eigenfunctions and eigenvalues of the covariance kernel~\cite{Rasmussen}.
Although the random variables~$\xi_k$ are formally defined,  there is in general no practical expression. 
For Gaussian processes, however, they are given by independent  normally distributed random variables~\cite{Rasmussen,S4}.
Therefore, Gaussian processes are often used, although it might be  problematic, when representing bounded physical processes. 

\emph{Time-varying} 
 stochastic dynamics  
\emph{cannot} be expressed directly in terms of  Karhunen-Lo\`{e}ve expansions, which require that the covariance structure is known a priori. For our random dynamics, however, this structure is given only implicitly by a hyperbolic boundary value problem. 
To generalize the Karhunen-Lo\`{e}ve expansion,  
we introduce a \textbf{generalized polynomial chaos} (gPC)  as a set of orthogonal subspaces
$$
\widehat{\S}_k \subseteq \mathbb{L}^2(\Omega,\mathbb{P})
\quad \text{with} \quad
\S_K \coloneqq \bigoplus\limits_{k=0}^K \widehat{\S}_k
\ \rightarrow \
\mathbb{L}^2(\Omega,\mathbb{P})
\quad \text{for} \quad
K \rightarrow \infty.
$$
These subspaces are spanned by orthogonal functions. 
Common choices for one-dimensional random variables~${\xi(\omega)\in\mathbb{R}}$ are the following polynomials~\cite{S4,S3}:

\vspace{4mm}
\noindent
\textbf{Legendre polynomials}
\tabto{5.5cm}
with uniform distribution 
$
\xi \sim \Uniform(-1,1)
$

\noindent
$\displaystyle
\phi_0(\xi) = 1,\quad
\phi_1(\xi) = \xi,\quad
\phi_{k+1}(\xi) = \frac{2k+1}{k+1} \xi \phi_{k}(\xi) - \frac{k}{k+1} \phi_{k-1}(\xi) \quad
$

\vspace{5mm}
\noindent
\textbf{Hermite polynomials}
\tabto{5.5cm}
with Gaussian distribution 
$\xi \sim \Normal(0,1)$
\vspace{1mm}

\noindent
$\displaystyle
\phi_0(\xi) = 1,\quad
\phi_1(\xi) = \xi, \quad
\phi_{k+1}(\xi) = \xi \phi_k(\xi) - k \phi_{k-1}(\xi) \quad
$

\vspace{5mm}
\noindent
Orthogonal functions with respect to multi-dimensional random parameters~${\xi(\omega)\in\mathbb{R}^{\NKL}}$ can be constructed by introducing a multi-index  ${\boldsymbol{k}\coloneqq(k_1,\ldots,k_M)^\T \in \K  }$ associated to an index set~${\K \subseteq\mathbb{N}_0^{\NKL}}$. 
Common choices are the index sets
\begin{alignat}{8}
& \K
&&= 
\big\{ 
\boldsymbol{k} \in \mathbb{N}_0^M \ \big| \
\lVert \boldsymbol{k} \rVert_0 \leq K
\big\}
&& \quad \text{for} \quad  {|\K | = (K+1)^{\NKL}}, \label{BasisFull}\tag{$\K_{\T}$} \\
& \K
&& =
\big\{ 
\boldsymbol{k} \in \mathbb{N}_0^M \ \big| \
\lVert \boldsymbol{k} \rVert_1 \leq K
\big\}
&& \quad \text{for} \quad 
|\K|=\frac{(\NKL+K )!} {\NKL! K !}. \label{BasisSparse}\tag{$\K_{\textup{S}}$}
\end{alignat}

\noindent
Since the random variables~${\xi_1,\ldots,\xi_M}$ are independent, the resulting polynomials
$$
\phi_{\boldsymbol{k}}(\xi)
\coloneqq
\phi_{k_1}(\xi_1) \cdot \ldots \cdot \phi_{k_M}(\xi_M)
\quad\text{satisfy}\quad
\langle \phi_{\boldsymbol{i}}, \phi_{\boldsymbol{j}}  \rangle_{\mathbb{P}}
=
\lVert \phi_{\boldsymbol{i}} \rVert^2_{\mathbb{P}}\,
\delta_{\boldsymbol{i},\boldsymbol{j}},
$$
where we write with abuse of notation~$
\big\langle \phi_{\boldsymbol{i}}, \phi_{\boldsymbol{j}}  \big\rangle_{\mathbb{P}}
=
\big\langle \phi_{\boldsymbol{i}} (\xi), \phi_{\boldsymbol{j}} (\xi) \big\rangle_{\mathbb{P}}
$.   
Then, a dynamic stochastic process~$\process(t,x;\xi)$ is approximated  by an orthogonal projection
\begin{equation}
\label{NgPC} \tag{$\mathbb{K}$\textup{gPC}}
\GK[\process](t,x;\xi) \coloneqq
\sum\limits_{\boldsymbol{k} \in \K }
\widehat{\process}_{\boldsymbol{k}}(t,x) 
\phi_{\boldsymbol{k}}(\xi)
\quad\text{with gPC modes} \quad
\widehat{\process}_{\boldsymbol{k}}(t,x) \coloneqq \frac{\big\langle \process(t,x;\cdot), \phi_{\boldsymbol{k}}(\cdot) \big\rangle_{\mathbb{P}}}{\lVert \phi_{\boldsymbol{k}} \rVert^2_{\mathbb{P}}}.
\end{equation}


\noindent
The truncated expansion~\eqref{NgPC} 
converges in the sense 
${
	\big\lVert \GK[\process](t,x;\cdot) - \process(t,x;\cdot) \big\rVert_{\mathbb{P}} \rightarrow 0
}$ 
for $K\rightarrow\infty$ 
provided that the probability measure satisfies  mild conditions, which hold for Legendre and Hermite polynomials~\cite{S2,funaro2008,courant89,Ullmann2012}.


\subsection{Stochastic Galerkin formulation}

The underlying system~\eqref{LyapunovDev1} is linear and hence the random fluctuations also satisfy the differential equation
\begin{equation}\label{EqDeviations}
\partial_t \Delta y(t,x;\xi) 
+
A(x;\xi)
\partial_x \Delta y(t,x;\xi) 
=
-S(x;\xi)
\Delta y(t,x;\xi). 
\end{equation}
The eigenvalue decomposition 
\begin{equation}\label{RandomEW}
A(x;\xi) = \Tsteady(x;\xi) \Lambda(x;\xi) \Tsteady^{-1}(x;\xi) 
\quad\text{with}\quad
\Lambda(x;\xi)
=
\diag\Big\{
\lambda^+(x;\xi),
\lambda^-(x;\xi)
\Big\}
\end{equation}
allows to rewrite this system in \textbf{Riemann coordinates}
\begin{equation}\label{RiemannInvariants}
\Riemann(t,x;\xi) \coloneqq
\begin{pmatrix} \Riemann^+(t,x;\xi) \\ \Riemann^-(t,x;\xi) \end{pmatrix}
\coloneqq \Tsteady^{-1}(x;\xi) \Delta y(t,x;\xi).
\end{equation}
Then, the linear balance law~\eqref{EqDeviations} with boundary conditions~\eqref{BC} is equivalent to the boundary value problem
\begin{equation}\label{CauchyZeta}
\begin{aligned}
&\partial_t \Riemann(t,x;\xi) + \Lambdasteady(x;\xi) \partial_x \Riemann(t,x;\xi) = - \SRiemann(x;\xi) \Riemann(t,x;\xi)\\
&\begin{pmatrix} \Riemann^+(t,0;\xi) \\ \Riemann^-(t,L;\xi) \end{pmatrix} 
=
\GRiemann
\begin{pmatrix} \Riemann^+(t,L;\xi) \\ \Riemann^-(t,0;\xi) \end{pmatrix},
\end{aligned}
\end{equation}
where the source term is defined as 
$$
	\SRiemann(x;\xi) \coloneqq \Tsteady^{-1}(x;\xi) S(x;\xi)   \Tsteady(x;\xi) + \Lambda(x;\xi) \Tsteady^{-1}(x;\xi) \partial_x \Tsteady(x;\xi).
$$
The boundary control is described by a matrix~$\GRiemann$, which states the transformed boundary conditions from equation~\eqref{BC}. 
We refer the reader to~\cite[Th.~6.1]{Gerster2019} for more details on this transform and we will state in Section~\ref{SubSecControlRule} an illustrative example. 
Similarly to Figure~\ref{FigNetwork0}, where the boundary control~$\Gphysicslinear$ is stated in physical quantities, Figure~\ref{FigNetwork1} illustrates that the boundary value problem~\eqref{CauchyZeta} is  extendable to multiple arcs.

\vspace{1cm}
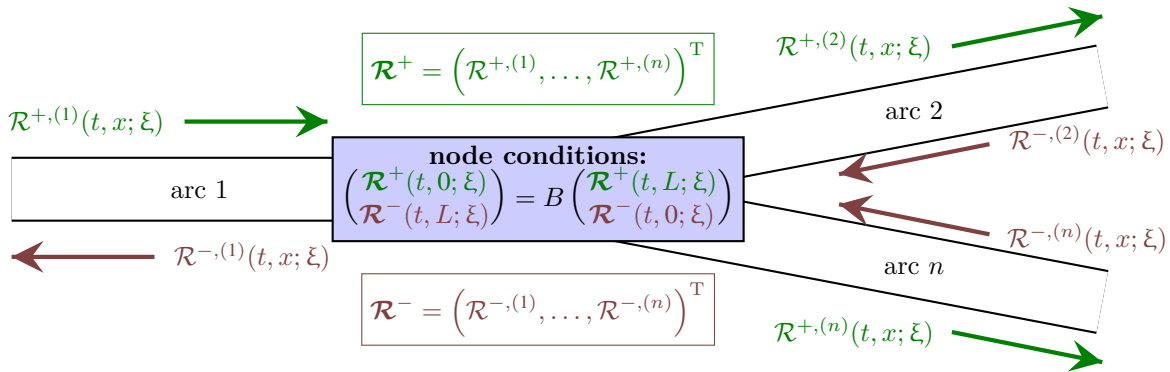
\begin{figure}[h]
	\tikzstyle{block} = [thick,draw, fill=blue!20, rectangle, 
minimum height=3em, minimum width=3em]
\tikzstyle{Kreis} = [thick,draw, fill=blue!20, circle, 
minimum height=3em, minimum width=3em]
\tikzstyle{sum} = [draw, fill=darkgreen!20, rectangle, 
minimum height=2em, minimum width=3em]
\tikzstyle{input} = [coordinate]
\tikzstyle{output} = [coordinate]
\tikzstyle{pinstyle} = [ pin edge={to-,thick,black}]
\tikzstyle{vecArrow} = [thick,double distance=23pt]

\vspace{-1cm}\hspace{-3mm}
\begin{tikzpicture}[scale=0.1]

\draw [vecArrow] (0,0) -- (50,0);
\draw [vecArrow] (145,15) -- (80,2.6);
\draw [vecArrow] (145,-15) -- (80,-2.6);

\node [block,align=center] at (70,0) {\textbf{node conditions:}\\
	$
\begin{pmatrix}\color{gruen} \boldsymbol{\mathcal{R}}^+(t,0;\xi) \\ \color{braun} \boldsymbol{\mathcal{R}}^-(t,L;\xi) \end{pmatrix} 
=
\GRiemann
\begin{pmatrix} \color{gruen} \boldsymbol{\mathcal{R}}^+(t,L;\xi) \\ \color{braun} \boldsymbol{\mathcal{R}}^-(t,0;\xi) \end{pmatrix} 
$
  };

\node at (25,0) {arc 1};
\node at (120,10.5) {arc 2};
\node at (120,-10.5) {arc $n$};

\node at (9.9,9) {\color{gruen}$\Riemann^{+,(1)}(t,x;\xi)$};
\node at (32,-9) {\color{braun}$\Riemann^{-,(1)}(t,x;\xi)$};

\node at (112,-19) {\color{gruen}$\Riemann^{+,(n)}(t,x;\xi)$};
\node at (143,6.5) {\color{braun}$\Riemann^{-,(2)}(t,x;\xi)$};

\node at (112,19) {\color{gruen}$\Riemann^{+,(2)}(t,x;\xi)$};
\node at (143,-6.5) {\color{braun}$\Riemann^{-,(n)}(t,x;\xi)$};

\node at (70,16) {\color{gruen}$
	\boxed{
	\boldsymbol{\mathcal{R}}^+= \Big(\Riemann^{+,(1)},\ldots,\Riemann^{+,(n)} \Big)^\T
}$};

\node at (70,-16) {\color{braun}$
	\boxed{
	\boldsymbol{\mathcal{R}}^-= \Big(\Riemann^{-,(1)},\ldots,\Riemann^{-,(n)} \Big)^\T
}
	$};

\draw [ultra thick, gruen, decoration={markings,mark=at position 1 with	{\arrow[scale=2,>=stealth]{>}}},postaction={decorate}] (	23,9) -- (42,9);
\draw [ultra thick, braun, decoration={markings,mark=at position 1 with	{\arrow[scale=2,>=stealth]{>}}},postaction={decorate}] (19,-9) -- (0,-9);

\draw [ultra thick, gruen, decoration={markings,mark=at position 1 with	{\arrow[scale=2,>=stealth]{>}}},postaction={decorate}]  (125,-19) -- (145,-23);

\draw [ultra thick, braun, decoration={markings,mark=at position 1 with	{\arrow[scale=2,>=stealth]{>}}},postaction={decorate}] (130,-6) -- (110,-2);

\draw [ultra thick, braun, decoration={markings,mark=at position 1 with	{\arrow[scale=2,>=stealth]{>}}},postaction={decorate}] (130,6) -- (110,2);

\draw [ultra thick, gruen, decoration={markings,mark=at position 1 with	{\arrow[scale=2,>=stealth]{>}}},postaction={decorate}]  (125,19) -- (145,23);

\end{tikzpicture}
	\caption{Network with $n$ arcs in Riemann coordinates with linear boundary conditions~$\GRiemann\in\mathbb{R}^{2n\times 2n}$.}
	\label{FigNetwork1}
\end{figure}

\noindent
We note that the eigenvalue decomposition~\eqref{RandomEW} must satisfy
\begin{equation}\label{RandomDistinctEW}
\lambda^-(x;\xi) < 0 < \lambda^+
(x;\xi)
\quad
\text{for all}
\quad
x\in[0,L]
\quad
\text{and}
\quad
\xi \sim \mathbb{P}
\end{equation}
such that the boundary control is applicable. 
In the following, we will transfer this property to a stochastic Galerkin formulation that expresses the random system~\eqref{CauchyZeta} as a sequence of deterministic problems. 
To this end, we replace all random quantities by the gPC approximation~\eqref{NgPC}, i.e.
\begin{alignat*}{8}
&\lambdasteady^\pm(x;\xi)
&&\approx
 \GK\big[\lambdasteady^\pm\big](x;\xi)
 &&=
\sum\limits_{\boldsymbol{k} \in \K }
\boldsymbol{\widehat{\lambda}^\pm_{k}}(x) 
\phi_{\boldsymbol{k}}(\xi), \\
&\SRiemann(x;\xi)
&&\approx
\GK\big[\SRiemann\big](x;\xi)
&&=
\sum\limits_{\boldsymbol{k} \in \K }
\SRhat_{\boldsymbol{k}}(x) 
\phi_{\boldsymbol{k}}(\xi), \\
&\Riemann(t,x;\xi)
&&\approx
\GK\big[\Riemann\big](t,x;\xi)
&&=
\sum\limits_{\boldsymbol{k} \in \K }
\Rhat_{\boldsymbol{k}}(t,x) 
\phi_{\boldsymbol{k}}(\xi).
\end{alignat*}

\noindent
The random system is projected onto the gPC basis~$(\phi_{\boldsymbol{k}})_{\boldsymbol{k}\in\K}$  by the Galerkin method
$$
\Big\langle
\partial_t \GK\big[\Riemann\big](t,x;\cdot) 
+ \GK\big[\Lambdasteady\big](x;\cdot) \,
\partial_x \GK\big[\Riemann\big](t,x;\cdot) 
+
\GK\big[\SRiemann\big](x;\cdot)
 \GK\big[\Riemann\big](t,x;\cdot),
\phi_{\boldsymbol{k}} (\cdot) \Big\rangle_{\mathbb{P}}=0
\quad\text{for all}\quad
\boldsymbol{k} \in \K.
$$
This leads to the \textbf{stochastic Galerkin formulation}
	\begin{align}\label{gPCPDE}
	&\partial_t \Rhat(t,x) + \Ahat(x) \partial_x\Rhat(t,x) =- \Shat(x)\Rhat(t,x)  \quad \text{for} \quad
	\Ahat
	\coloneqq 
	\begin{pmatrix}
	\boldsymbol{\widehat{A}^+} \\
	&\boldsymbol{\widehat{A}^-}
	\end{pmatrix}\\ 
	%
	&
	\begin{aligned}
	&\text{with} \quad
	&\boldsymbol{\widehat{A}^\pm}(x)
	&\coloneqq
	\sum\limits_{\boldsymbol{k} \in \K }
	\boldsymbol{\widehat{\lambda}^\pm_{k}}(x) 
	\Big( \langle
\phi_{\boldsymbol{k}},
	\phi_{\boldsymbol{i}} \phi_{\boldsymbol{j}}
	\rangle_{\mathbb{P}} \Big)_{\boldsymbol{i},\boldsymbol{j}\in\K}\\
	&\text{and} \quad
	&\Shat(x)
	&\coloneqq
	\sum\limits_{\boldsymbol{k} \in \K }
\SRhat_{\boldsymbol{k}}(x) 
\Big( \langle
\phi_{\boldsymbol{k}},
\phi_{\boldsymbol{i}} \phi_{\boldsymbol{j}}
\rangle_{\mathbb{P}} \Big)_{\boldsymbol{i},\boldsymbol{j}\in\K}.
	\end{aligned} \nonumber
	\end{align}

\noindent
Since the matrices~$\boldsymbol{\widehat{A}^\pm}(x)$ are symmetric, there exists an orthogonal eigenvalue decomposition
$$
	\boldsymbol{\widehat{A}^\pm}(x) = \boldsymbol{\widehat{\mathcal{T}}^\pm}(x) \boldsymbol{\widehat{\mathcal{D}}^\pm}(x) \boldsymbol{\widehat{\mathcal{T}}^\pm}(x)^{\T}.
$$
This allows to diagonalize the stochastic Galerkin formulation~\eqref{gPCPDE} by introducing the second class of Riemann invariants
\begin{equation}\label{TrafoCeta}
\Rd(t,x) =
\boldsymbol{\widehat{\mathcal{T}}}(x)^{\textup{T}} \Rhat(t,x)
\quad\text{for}\quad
\boldsymbol{\widehat{\mathcal{T}}}(x)
\coloneqq
\diag
\Big\{
\boldsymbol{\widehat{\mathcal{T}}}^+(x),\boldsymbol{\widehat{\mathcal{T}}}^-(x)\Big\}.
\end{equation}

\noindent
These Riemann invariants are described by the boundary value problem
\begin{align}
&\partial_t \Rd(t,x) + \Dhat(x) \partial_x \Rd(t,x)  = - \Bhat(x) \Rd(t,x), \label{IBVPzeta}\\
&\begin{pmatrix} \ZetaPlus(t,0) \\ \ZetaMinus(t,L) \end{pmatrix} 
=
\Ghat \begin{pmatrix} \ZetaPlus(t,L) \\ \ZetaMinus(t,0) \end{pmatrix} \nonumber\\
&
\text{for}\quad
\Ghat
\coloneqq
\begin{pmatrix}
\boldsymbol{\widehat{\mathcal{T}}^+}(0)\\
&\boldsymbol{\widehat{\mathcal{T}}^-}(L)
\end{pmatrix}^\T
\begin{pmatrix}
B_{1,1} \indikator & B_{1,2} \indikator  \\
B_{2,1} \indikator & B_{2,2} \indikator  \\
\end{pmatrix} 
\begin{pmatrix}
\boldsymbol{\widehat{\mathcal{T}}^+}(L)\\
&\boldsymbol{\widehat{\mathcal{T}}^-}(0)
\end{pmatrix},
 \nonumber 
 \nonumber
\end{align}
where~${\indikator\in\mathbb{R}^{|\K|\times|\K| }}$ denotes the identity matrix. 
The characteristic speeds and the source term read as
\begin{equation*}
\Dhat(x) \coloneqq \diag\Big\{
\boldsymbol{\widehat{\mathcal{D}}^+}(x),\boldsymbol{\widehat{\mathcal{D}}^-}(x)
\Big\}
\quad\text{and}\quad
\Bhat(x)
\coloneqq
\boldsymbol{\widehat{\mathcal{T}}}(x)^\T
\SRhat(x)
\boldsymbol{\widehat{\mathcal{T}}}(x)
+
\Dhat(x)
\boldsymbol{\widehat{\mathcal{T}}}(x)^\T
\partial_x
\boldsymbol{\widehat{\mathcal{T}}}(x).
\end{equation*}

\noindent
Note that solutions to the  systems~\eqref{LyapunovDev1} and~\eqref{EqDeviations}  can also be expanded in terms of polynomial chaos expansions. However, it has been observed in~\cite[Sec.~4.2]{H4} that  the  basis~\eqref{BasisSparse}  may lead to a loss of hyperbolicity. 
In contrast, the stochastic Galerkin formulation~\eqref{gPCPDE} is always hyperbolic. 
To ensure the wellposedness of the  boundary value problem~\eqref{IBVPzeta}, we have to guarantee
\begin{equation}\label{DistinctDeterministicEW}
\boldsymbol{\widehat{\mathcal{D}}^-}(x)<0<\boldsymbol{\widehat{\mathcal{D}}^+}(x)
\quad
\text{for all}
\quad  x\in[0,L].
\end{equation}
Indeed, 
\cite[Th.~2]{Sonday2011} and \cite[Th.~2.1]{H3} state that property~\eqref{RandomDistinctEW} 
implies~\eqref{DistinctDeterministicEW}. Namely, the matrices~$\boldsymbol{\widehat{A}^\pm}$ are strictly positive and negative definite, respectively, since for all
${
	\widehat{y}\neq (0,\ldots,0)^\T 
}$ 
and basis functions~$\phi_{\boldsymbol{k}}$ we have
	$$
	\widehat{y}^\T \boldsymbol{\widehat{A}^\pm}\widehat{y}
	=
	\pm
	\int
	\bigg(
	\sqrt{\big| 
		\GK[\lambdasteady^\pm](x;\xi) 
		\big|}
	\sum\limits_{\boldsymbol{k}\in\K}
	\widehat{y}_{\boldsymbol{k}} \phi_{\boldsymbol{k}}(\xi)
	\bigg)^2
	\d \mathbb{P}.
	$$

\noindent
The wellposedness of the boundary value problem~\eqref{IBVPzeta} follows now directly from~\cite[Th.~A.4]{O1}. More precisely,~\cite[Th.~A.4]{O1} states that the $L^2$-solution 
${
	\Rd \, : \, \mathbb{R}^+_0\times (0,L) \rightarrow  \mathbb{R}^{2|\K|}
}$ 
to the boundary value problem~\eqref{IBVPzeta} 
with initial values~${\Rd(0,\cdot)\in L^2\big((0,L);\mathbb{R}^{2|\K|} \big)}$ 
is a continuous map 
\begin{equation}\label{DeterministicL2Solution}
\Rd \ \ : \ \ \mathbb{R}^+_0 \rightarrow L^2\Big((0,L);\mathbb{R}^{2|\K|}\Big) , \ \ \
t \mapsto \Rd(t,\cdot)
\end{equation}
that is unique and that depends also continuously on initial data. Due to the underlying tensor product structure and the relation~\eqref{TrafoCeta}, i.e.~${
	\Rhat(t,x)
	=
	\boldsymbol{\widehat{\mathcal{T}}}(x)
	\Rd(t,x)
}$, we can  interpret the gPC expansion
\begin{equation}\label{RandomL2Solution}
\GK\big[\Riemann\big]
\ \ : \ \ \mathbb{R}^+_0 \rightarrow
L^2\Big( (0,L);\mathbb{R}^2 \Big) \bigotimes \mathbb{L}^2(\Omega,\mathbb{P}),  \ \ \
t \mapsto \GK\big[\Riemann\big](t,\cdot;\cdot)
=
\sum\limits_{\boldsymbol{k} \in \K }
\Rhat_{\boldsymbol{k}}(t,\cdot) 
\phi_{\boldsymbol{k}}(\cdot)
\end{equation}
as a well-defined \emph{random} $L^2$-solution.


\begin{figure}[h]
		\begin{tikzpicture}[scale=1.6,cap=round]
	
\begin{scope}

\tikzstyle{important line}=[very thick]

\filldraw[fill=white!80!gruen,draw=none] (0,0) -- (1,0.5) arc(30:75:1);
\filldraw[fill=white!80!myred,draw=none] (0,0) -- (-0.2,1) arc(100:130:1);
\filldraw[fill=white!80!gruen,draw=none] (3,0) -- (2.6,1.1) arc(110:180:1);
\filldraw[fill=white!80!myred,draw=none] (3,0) -- (4,0.3) arc(3:70:0.6);

\draw[style=important line] (-0.2,0) -- (3.2,0);
\draw[style=important line,->] (0,-0.2) -- (0,1.4);
\draw[style=important line,->] (3,-0.2) -- (3,1.4);
\draw (0,-0.4) -- (0,-0.4) node {$x=0$};
\draw (3,-0.4) -- (3,-0.4) node {$x=L$};
\draw (1.5,1.85) -- (1.5,1.85) node {\underline{random boundary value problem~\eqref{CauchyZeta}}};


\draw (1.1,1.2) -- (1.1,1.2) node {\color{gruen}$\Riemann^+(t,0;\xi)>0$};
\draw (1.9,0.4) -- (1.9,0.4) node {\color{gruen}$\Riemann^-(t,L;\xi)<0$};
 

%
%


\renewcommand{\arraystretch}{1.25}
\draw (1.47,-1.2) -- (1.47,-1.2) node {
	$\displaystyle
	\boxed{
		\begin{pmatrix} \color{gruen}\Riemann^+(t,0;\xi) \\ \color{gruen}\Riemann^-(t,L;\xi) \end{pmatrix} 
		=
		\GRiemann \begin{pmatrix} 
		\color{myred}\Riemann^+(t,L;\xi) \\ 
		\color{myred}\Riemann^-(t,0;\xi) \end{pmatrix} 
	}	$
};


\draw (4.05,0.2) -- (4.05,0.2) node
{\color{myred}$\Riemann^+(t,L;\xi)>0$};

\draw (-0.35,1.5) -- (-0.35,1.5) node {time $t$};

\end{scope}


\begin{scope}[xshift=5.5cm]

\tikzstyle{important line}=[very thick]


\draw[style=important line] (-0.2,0) -- (3.2,0);
\draw[style=important line,->] (0,-0.2) -- (0,1.4);
\draw[style=important line,->] (3,-0.2) -- (3,1.4);

\draw (0,-0.4) -- (0,-0.4) node {$x=0$};
\draw (3,-0.4) -- (3,-0.4) node {$x=L$};


\draw[thick,gruen] (0,0) -- (0.93,0.53);
\draw[thick,gruen] (0,0) -- (0.75,0.75);
\draw[thick,gruen] (0,0) -- (0.5,0.9);
\draw[thick,dashed,myred] (0,0) -- (-0.25,0.98);
\draw[thick,dashed,myred] (0,0) -- (-0.40,0.95);
\draw[thick,dashed,myred] (0,0) -- (-0.55,0.90);

\draw[thick,gruen] (3,0) -- (2,0.45);
\draw[thick,gruen] (3,0) -- (2.2,0.8);
\draw[thick,gruen] (3,0) -- (2.5,1);
\draw[thick,dashed,myred] (3,0) -- (3.99,0.35);
\draw[thick,dashed,myred] (3,0) -- (3.85,0.6);
\draw[thick,dashed,myred] (3,0) -- (3.6,0.8);

\draw (0.7,1.2) -- (0.7,1.2) node {\color{gruen}$\ZetaPlus(t,0) >0$};
\draw (2.3,1.2) -- (2.3,1.2) node {\color{gruen}$\ZetaMinus(t,L) <0$};


\draw (-0.8,1.2) -- (-0.8,1.2) node {\color{myred}$\ZetaMinus(t,0) <0$};

\draw (1.5,1.85) -- (1.5,1.85) node {\underline{deterministic boundary value problem~\eqref{IBVPzeta}}};

\draw (-0.35,1.5) -- (-0.35,1.5) node {time $t$};

\draw (1.47,-1.2) -- (1.47,-1.2) node {
	$\displaystyle
	\boxed{
		\begin{pmatrix} \color{gruen}\ZetaPlus(t,0) \\ \color{gruen}\ZetaMinus(t,L) \end{pmatrix} 
		=
		\Ghat\begin{pmatrix} \color{myred} \ZetaPlus(t,L) \\ \color{myred} \ZetaMinus(t,0) \end{pmatrix} 
	}
	$
};

\end{scope}








\end{tikzpicture}
	\caption{Random (left) and deterministic (right) boundary value problems~\eqref{CauchyZeta} and~\eqref{IBVPzeta} for~$|\K|=3$. Characteristic speeds that point into the spatial domain are shown in green and those pointing out are shown in red, respectively.}
	\label{FIGrandomBC}
\end{figure}
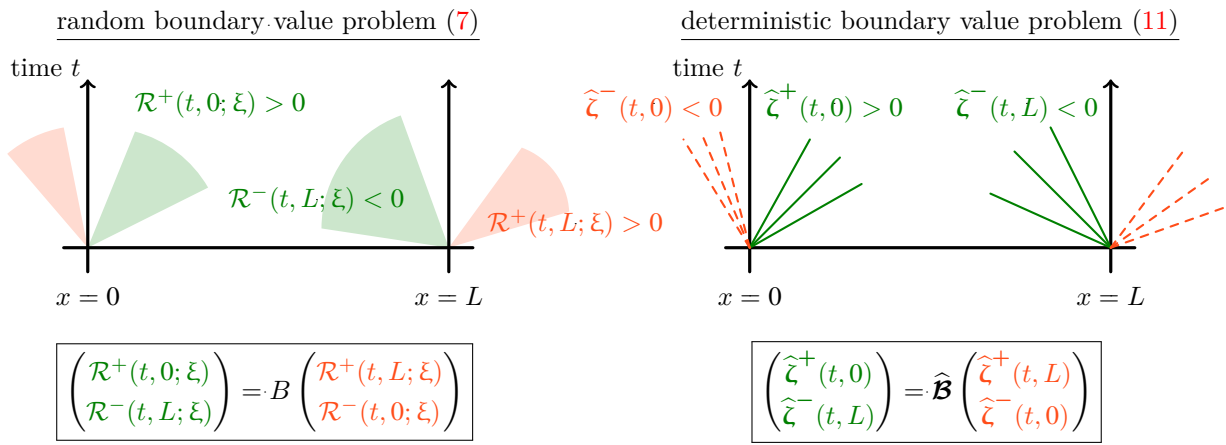

\noindent
Figure~\ref{FIGrandomBC} summarizes the previous analysis in terms of specified boundary conditions. 
The left panel shows 
for the random boundary value problem~\eqref{CauchyZeta} 
the domain of possible realizations of 
the characteristic curves, along which the  realizations~$\Riemann^\pm\big(t,x;\xi(\omega)\big)$ 
are constant. 
The domain corresponding to those curves that point into the spatial domain and influence the solution is shown in green. The domain corresponding to those pointing out is shown in red, respectively.
Likewise, the deterministic solution~$\Rd(t,x)$ to the system~\eqref{IBVPzeta} is constant on the set of characteristic curves as illustrated in the right panel. 
The corresponding boundary conditions, illustrated below in Figure~\ref{FIGrandomBC}, specify the new values (green) that propagate from the boundary into the domain in terms of those that go out of the domain (red).

The consideration of all possible realizations gives rise to a~\emph{domain}, where characteristic curves are located. 
This poses a challenge in the analysis, since  
the feasibility of the boundary control must be ensured for all realizations~\cite{Schuster2019} and an appropriate stabilization concept must be introduced for the random systems.
This challenge can be partially circumvented through a concept based on 
a function space of view, which is motivated by the fact that 
the $L^2$-solutions~\eqref{RandomL2Solution} and \eqref{DeterministicL2Solution} contribute to the mean squared error~\eqref{MSE} the same information. 
The solution~\eqref{DeterministicL2Solution}  results for~${K\rightarrow\infty}$ in  infinite, but~\emph{countably many} characteristic curves. 
We will show in the following section that this allows for a deterministic stability analysis.



\begin{remark}\label{Remark1}
	We remark that condition~\eqref{RandomDistinctEW} on the random characteristic speeds is sufficient, but not necessary. 	Consider a polynomial chaos approximation~\eqref{NgPC} with Hermite polynomials and  truncation~${K=1}$ that contains the Karhunen-Lo\`{e}ve decomposition~\eqref{KL} for a Gaussian process as special case~\cite{Rasmussen,S4}, i.e.
	\begin{equation}\label{KLvsPC}
	\GK\big[\lambdasteady^+\big]\big(x;\xi(\omega)\big)
	=
	\mathcal{K}\big[\lambdasteady^+\big](x;\omega)
	=
	\mathbb{E}\big[
	\lambdasteady^+(x;\xi)\big] 
	+
	\sum\limits_{k=1}^{\NKL}
	\sqrt{\lambdaKL_k} \psi_k(x)
	\xi_k(\omega).
	\end{equation}
	By exploiting the orthogonality $\langle \xi_k,\xi_\ell \rangle_{\mathbb{P}}=\delta_{k,\ell}$  
	and by using a lexicographic ordering~\cite[Ch.~5]{S17} of the multi-indices~${\boldsymbol{k}\in\K}$ we obtain 
	the matrix~$\boldsymbol{\widehat{A}^+}$ in the stochastic Galerkin formulation~\eqref{gPCPDE} as
	\renewcommand{\arraystretch}{1.2}
	$$
	\boldsymbol{\widehat{A}^+}
	\coloneqq
	\begin{pmatrix}
	\mathbb{E}\big[
	\lambdasteady^+(x;\xi)\big] 
	& \sqrt{\lambdaKL_1} \psi_1(x)
	&\cdots
	& \sqrt{\lambdaKL_{\NKL}} \psi_{\NKL}(x) \\
	\sqrt{\lambdaKL_{1}} \psi_{1}(x) & \mathbb{E}
	\big[
	\lambdasteady^+(x;\xi)\big] \\
	\vdots & & \ddots \\
	\sqrt{\lambdaKL_{\NKL}} \psi_{\NKL}(x) 
	& & &\mathbb{E}\big[\lambdasteady^+(x;\xi)\big] 
	\end{pmatrix} \in \mathbb{R}^{(M+1)\times(M+1)}. 
	$$
	
	\renewcommand{\arraystretch}{1}

	\noindent
	It is strictly positive definite according to Gershgorin circle theorem provided that 
	\begin{equation}\label{AssumptionKL}
	\mathbb{E}
	\big[
	\lambdasteady^+(x;\xi)\big] 
	>
	\sum\limits_{k=1}^{\NKL}
	\Big| \sqrt{\lambdaKL_{k}} \psi_{k}(x) \Big|
	\end{equation}
	is satisfied. This assumption is  not restrictive if the deviations from the mean of the stochastic process are sufficiently small. 
	We note that the mean and the variance are directly given by the gPC modes for normalized gPC expansions as
	\begin{align*}
	\mathbb{E}\Big[
	\GK\big[
	\lambdasteady^+\big](x;\xi) 
	\Big]
	&=
	\boldsymbol{\widehat{\lambda}^+}_{\zeros}(x),
	\quad
	\boldsymbol{\zeros}\coloneqq
	(0,\ldots,0)^\T,\\
	\mathbb{V}\Big[
	\GK\big[
	\lambdasteady^+\big](x;\xi) 
	\Big]
	&=
	\mathbb{E}\Big[
	\GK\big[
	\lambdasteady^+\big](x;\xi)^2
	\Big]
	-
	\mathbb{E}\Big[
	\GK\big[
	\lambdasteady^+\big](x;\xi) 
	\Big]^2 \\
	&=
	\sum\limits_{\boldsymbol{k},\boldsymbol{\ell} \in \K}
	\boldsymbol{\widehat{\lambda}^+_{k}}(x)
	\boldsymbol{\widehat{\lambda}^\pm_{\ell}}(x)
	\langle \phi_{\boldsymbol{k}}, \phi_{\boldsymbol{\ell}} \rangle_{\mathbb{P}}
	-
	\boldsymbol{\widehat{\lambda}^+_{\zeros}}(x)^2 \\
	&=
	\sum\limits_{\boldsymbol{k} \in \K\backslash\{
		\boldsymbol{\zeros} \} }
	\boldsymbol{\widehat{\lambda}^+_{k}}(x)^2.
	\end{align*}
	Therefore, 
	property~\eqref{AssumptionKL} yields
	\begin{align*}
	\mathbb{E}
	\Big[
	\mathcal{K}\big[\lambdasteady^+\big](x;\omega) \Big]
	&>
	\sum\limits_{k=1}^{\NKL}
	\big| \sqrt{\lambdaKL_{k}} \psi_{k}(x) \big|
	\geq
	\sqrt{
		\sum\limits_{k=1}^{\NKL}
		\lambdaKL_{k} \psi_{k}^2(x)
	} 
	=
	\mathbb{V}
	\Big[
	\mathcal{K}\big[\lambdasteady^+\big](x;\omega) \Big]^{\nicefrac{1}{2}}
	\end{align*}
	and the coefficient of variation (\textup{CV}) must satisfy
	$$
	\textup{CV}
	\Big[
	\mathcal{K}\big[\lambdasteady^+\big](x;\omega) \Big]
	\coloneqq
	\frac{
		\mathbb{V}
		\Big[
		\mathcal{K}\big[\lambdasteady^+\big](x;\omega) \Big]^{\nicefrac{1}{2}}	
	}{
		\mathbb{E}
		\Big[
		\mathcal{K}\big[\lambdasteady^+\big](x;\omega) \Big]
	}
	<1
	\quad\text{for all}\quad 
	x\in[0,L].
	$$

	\noindent
	In fact, the truncation~${K<\infty}$ acts as a dimension reduction that makes the necessary conditions less restrictive. 
	However, we note that for sufficiently large~${K\in\mathbb{N}_0}$ the matrix~$
	\boldsymbol{\widehat{A}^+}
	$ has both positive and negative eigenvalues if the random variable~$\GK\big[ \lambdasteady^+(x;\xi)\big]$ has both 
	positive and negative realizations~\cite[Th.~2.1]{H3}. 
	Since the Gaussian distribution involves realizations on the whole real line,  the $\gamma$-distribution has been proposed in~\cite{S19,S17} as a ``truncated Gaussian model''. 
	Then, positivity can be ensured for all realizations and for all truncations~${K\in\mathbb{N}_0}$.

\end{remark}

\section{Lyapunov stability analysis}
\label{Sec4}

We design a boundary control to diminish the mean squared error~\eqref{MSE} exponentially fast over time. Due to relation~\eqref{RiemannInvariants} it suffices to guarantee that the stochastic Riemann invariants decay exponentially fast, i.e.
$$
 \mathbb{E} \bigg[
\big\lVert \GK\big[\Riemann\big](t,\cdot;\xi)
\big\rVert^2_{L^2} \bigg]
\leq
c\, e^{-\mu t} \,
 \mathbb{E} \bigg[
\big\lVert \GK\big[\Riemann\big](0,\cdot;\xi)
\big\rVert^2_{L^2} \bigg]
$$
with decay rate~${\mu>0}$ and positive constant~${c>0}$. 
By exploiting the orthogonality 
$
\boldsymbol{\widehat{\mathcal{T}}}^{\textup{T}}
=
\boldsymbol{\widehat{\mathcal{T}}}^{-1}
$ 
we obtain the relation
\begin{equation}\label{stochasticStability}
 \mathbb{E} \bigg[
\big\lVert \GK\big[\Riemann\big](t,\cdot;\xi)
\big\rVert^2_{L^2} \bigg]
=
\big\lVert
\Rhat(t,\cdot)\big\rVert^2_{L^2}
= 
\big\lVert
\boldsymbol{\widehat{\mathcal{T}}}(\cdot)
\Rd(t,\cdot)
\big\rVert^2_{L^2}
=
\big\lVert
\Rd(t,\cdot)
\big\rVert^2_{L^2},
\end{equation}
where the last equality makes the~$L^2$-norm essential. 
We observe that the stochastic stabilization follows from the deterministic stabilization of the augmented system~\eqref{IBVPzeta}. 
These quantities do not have a physical meaning. 
However, the derived control will be formulated in the original quantities, which leads to an explicit and implementable feedback control.  
The proof of the following theorem is a modification of the Lyapunov stability analysis in~\cite[Sec.~5]{O1}.


\begin{theorem}\label{TheoremStability}

Define the weights~${
W(x)\coloneqq \diag \big\{ W^+(x),W^-(x)  \big\}}$, 
${
W^\pm(x) \coloneqq \diag \big\{ w_1^\pm(x),\ldots,w_{|\K|}^\pm(x)\big\}
}$ by
\begin{equation}\label{Weights}
\begin{aligned}
w^+_{k}(x)
&\coloneqq
\frac{h^+_{k}}{ \DhatPlus_{k} (x)} \textup{exp}\bigg( -\mub \int_{0}^{x} \frac{1}{\DhatPlus_{k}(s)} \d s \bigg), 
\\
w_{k}^-(x)
&\coloneqq
\frac{h_{k}^-}{ \big|\DhatMinus_{k}(x)\big|} \textup{exp}\bigg( \mub \int_{x}^{L} \frac{1}{\DhatMinus_{k}(s)} \d s \bigg) \\
\end{aligned}
\end{equation}
for $k=1,\ldots,|\K|$ and assume given positive values~${\mub,h^+_k, h^-_k>0}$ that satisfy the \textbf{dissipativity condition}
\begin{align}
&e^{\mub \frac{L}{2 \lambda_{\min} }} \big\lVert \mathcal{D} \Ghat \mathcal{D}^{-1} \big\lVert_2\leq 1  \tag{$\mathcal{D}$}\label{ThInequality2} \\
&\begin{aligned}
&	\text{for } \ \
&	\lambda_{\min} &\coloneqq 
\min\limits_{\stackrel{x \in [0,L],}{k=1,\ldots,|\K|}}
\Big\{ \DhatPlus_{k}(x) ,\big|\DhatMinus_{k}(x)\big| \Big\} \\
&\text{and}\quad
&	\mathcal{D} &\coloneqq \diag\Big\{
h^+_1,\ldots,h^+_{|\K|},\ h^-_1,\ldots,h^-_{|\K|}
\Big\}.
\end{aligned}
\nonumber
\end{align}
	Here,~$\lVert A \rVert_2\coloneqq 
	\largestEV
	\big\{ A^\T A \big\}^{\nicefrac{1}{2}}
	$
	denotes the~spectral norm of a matrix  with largest absolute eigenvalue~$\largestEV$.  
Then, the mean squared 
 gPC expansion, corresponding to the boundary value problem~\eqref{CauchyZeta},  satisfies
\begin{align}
&\mathbb{E} \bigg[
\Big\lVert \GK\big[\Riemann\big](t,\cdot;\xi)
\Big\rVert^2_{L^2} \bigg]
\leq
c\, e^{-\mu t}
\,
\mathbb{E} \bigg[
\Big\lVert \GK\big[\Riemann\big](0,\cdot;\xi)
\Big\rVert^2_{L^2} \bigg] \nonumber \\
&\text{for}\quad
\mu=
\mub 
+ 
\min_{x \in [0,L]}\Big\{ \smallestEV \big\{ 
W(x)\Bhat(x)+\Bhat(x)^{\T}W(x) 
\big\} \Big\},
\label{guaranteedDecayRate}
\end{align}
where~$\smallestEV$ denotes the smallest eigenvalue of the space-varying, symmetric matrix~$
W\Bhat+\Bhat^{\T}W$.


\end{theorem}

\begin{proof}[Proof (Similarly to~{\cite[Sec.~5]{O1}})]
	
The weights~\eqref{Weights} define the Lyapunov function
\begin{equation}\label{Lyapunov}
\L(t) \coloneqq \int_0^L \Rd(t,x)^{\T} W(x) \Rd(t,x) \d x.
\end{equation}
To prove its exponential decay, we define the matrices
\begin{align}
\H &\coloneqq
\Ghat^{\T} 
\begin{pmatrix}
W^+(0) \boldsymbol{\widehat{\mathcal{D}}}^+(0) \\ & W^-(L) \big| \boldsymbol{\widehat{\mathcal{D}}}^-(L) \big|  
\end{pmatrix}
\Ghat 
 -
\begin{pmatrix}
W^+(L) \boldsymbol{\widehat{\mathcal{D}}}^+(L) \\ & W^-(0) \big| \boldsymbol{\widehat{\mathcal{D}}}^-(0) \big|  
\end{pmatrix}, \label{HLyapunov} \\ 
\M(x) &\coloneqq
-\frac{\partial}{\partial x}
\Big( W(x)\boldsymbol{\widehat{\mathcal{D}}}(x) \Big)
+W(x)\Bhat(x)+\Bhat(x)^{\T}W(x). \label{MLyapunov}
\end{align}

\noindent
It has been shown in~\cite{Gerster2019} that the matrix~$\mathcal{H}$ is negative semidefinite if inequality~\eqref{ThInequality2} holds.  
For now, we assume 
\begin{equation} \label{ToStrongRegularity}
\Rd \in C^1\Big([0,\infty) \times [0,L];\mathbb{R}^{2|\K|}\Big).
\end{equation}
Since the matrix~$2W\Bhat$ and the symmetric matrix~$W\Bhat+\Bhat^{\T}W$ represent the same quadratic form, i.e.
$$
\Rd(t,x)^\T\ W(x)\Bhat(x) \ \Rd(t,x)
=
\Rd(t,x)^\T \
\frac{
W(x)\Bhat(x)+\Bhat(x)^{\T}W(x)}{2}\
\Rd(t,x),
$$
the time derivative of the Lyapunov function is
\begin{alignat}{8}
& \L'(t)
&&=&& 2 \int_0^L \Rd(t,x)^{\T}W(x)\partial_t \Rd(t,x) \d x \nonumber \\
& &&= 
-&&2 \int_0^L \Rd(t,x)^{\T}W(x)\boldsymbol{\widehat{\mathcal{D}}}(x)\partial_x\Rd(t,x) \d x 
-2 \int_0^L \Rd(t,x)^{\T}W(x)\Bhat(x)\Rd(t,x) \d x \nonumber \\
& &&= &&
\int_0^L \Rd(t,x)^{\T} \bigg[
\frac{\partial}{\partial x} \Big( W(x)\boldsymbol{\widehat{\mathcal{D}}}(x) \Big)
-
W(x)\Bhat(x)-\Bhat(x)^{\T}W(x)
\bigg] 
\Rd(t,x) \d x \nonumber \\
& && &&- \Big[
\Rd(t,L)^{\T}W(L)\boldsymbol{\widehat{\mathcal{D}}}(L) \Rd(t,L)
-
\Rd(t,0)^{\T}W(0)\boldsymbol{\widehat{\mathcal{D}}}(0) \Rd(t,0)
\Big]. \label{IntroStabiBCgeneral1}
\end{alignat}

\noindent
Using the linear feedback boundary conditions, the boundary terms~\eqref{IntroStabiBCgeneral1} read as

\begin{alignat*}{8}
& &&-&& \ \Big[
\Rd(t,L)^{\T}W(L)\boldsymbol{\widehat{\mathcal{D}}}(L) \Rd(t,L)
\, - \,
\Rd(t,0)^{\T}W(0)\boldsymbol{\widehat{\mathcal{D}}}(0) \Rd(t,0) \Big] \\
&= && &&
\begin{pmatrix} \Rd^+(t,0) \\ \Rd^-(t,L) \end{pmatrix}^{\T}
\begin{pmatrix}
W^+(0)\boldsymbol{\widehat{\mathcal{D}}}^+(0) \\ & - W^-(L)\boldsymbol{\widehat{\mathcal{D}}}^-(L)
\end{pmatrix}	
\begin{pmatrix} \Rd^+(t,0) \\ \Rd^-(t,L) \end{pmatrix}\\
& &&- &&
\begin{pmatrix} \Rd^+(t,L) \\ \Rd^-(t,0) \end{pmatrix}^{\T}
\begin{pmatrix}
W^+(L)\boldsymbol{\widehat{\mathcal{D}}}^+(L) \\ & - W^-(0)\boldsymbol{\widehat{\mathcal{D}}}^-(0)
\end{pmatrix}	
\begin{pmatrix} \Rd^+(t,L) \\ \Rd^-(t,0) \end{pmatrix} \\
&= && &&
\begin{pmatrix} \Rd^+(t,L) \\ \Rd^-(t,0)  \end{pmatrix}^{\T} 
\H
\begin{pmatrix} \Rd^+(t,L) \\ \Rd^-(t,0)  \end{pmatrix}.
\end{alignat*}

\noindent
We recall that $\smallestEV$ denotes the smallest eigenvalue of a matrix  and we recall that 
the matrix~$\H$ is negative semidefinite. Then, we obtain the estimates
\begin{align*}
\L'(t)
&=
\begin{pmatrix} \Rd^+(t,L) \\ \Rd^-(t,0)  \end{pmatrix}^{\T} 
\H
\begin{pmatrix} \Rd^+(t,L) \\ \Rd^-(t,0)  \end{pmatrix}
-
\int_0^L \Rd(t,x)^{\T} \M(x) \Rd(t,x)  \d x \\
&\leq-
\min_{x \in [0,L]}\Big\{ \smallestEV \big\{ W^{-\nicefrac{1}{2}}(x) \M(x) W^{-\nicefrac{1}{2}}(x) \big\} \Big\}
\L(t) \\
\Longrightarrow \quad
\L(t) & \leq e^{-\mu t} \L(0),\quad
\mu \coloneqq \min_{x \in [0,L]}\Big\{ \smallestEV \big\{ W^{-\nicefrac{1}{2}}(x) \M(x) W^{-\nicefrac{1}{2}}(x) \big\} \Big\}.
\end{align*}
Due to~${
	\partial_x
\big( W(x)\boldsymbol{\widehat{\mathcal{D}}}(x) \big)
=-\mub W(x)
}$ 
the guaranteed decay rate~\eqref{guaranteedDecayRate} is obtained. \\

\noindent
So far, we have assumed 
differentiable solutions~\eqref{ToStrongRegularity} 
that satisfy boundary conditions. 
By following the approach in~\cite[Sec.~2.1]{O1}, the previous analysis can be extended to general $L^2$-solutions. 
Since initial values ${\Rd(0,x) \in C^1\big( (0,L);\mathbb{R}^{2|\K|} \big) }$
	are dense in~$L^2\big( (0,L);\mathbb{R}^{2|\K|} \big)$, 
 there is a differentiable sequence~
$$
\Rd^{(j)}(0,x)
 \in C^1\Big( (0,L);\mathbb{R}^{2|\K|} \Big) 
$$
that converges to initial values~$\Rd(0,x)$ in~$L^2\big( (0,L);\mathbb{R}^{2|\K|} \big) $. 
This sequence vanishes for~${x=0}$,  ${x=L}$ and satisfies the linear boundary conditions. It has been shown in~\cite[Th.~A.1]{O1} that also the solution for~${t>0}$ is differentiable and satisfies
	\begin{equation*}
	\Rd^{(j)} \in  
	C^1\Big( [0,\infty);L^2\big( (0,L)^{2|\K|} \big) \Big) 
	\bigcap
	C^0\Big( [0,\infty);H^1\big( (0,L)^{2|\K|} \big) \Big),
	\end{equation*}
	where $H^1$ denotes a Sobolev space. 
	This regularity is sufficient to obtain the sequence of estimates
	$$
	\L^{(j)}(t) \leq e^{-\mu t} \L^{(j)}(0) 
	\quad \text{for} \quad
	\L^{(j)}(t) \coloneqq \int_0^L \Rd^{(j)}(t,x)^{\T} W(x) \Rd^{(j)}(t,x) \d x, 
	$$
which implies~${\L(t) \leq e^{-\mu t} \L(0) }$ for~${j\rightarrow \infty}$.

\end{proof}

\noindent
Finally, we state a sufficient condition to ensure the dissipativity condition, which can be verified analytically.

\begin{corollary}\label{CorDissipativityCond}
	The dissipativity condition~\eqref{ThInequality2} holds if the boundary control~$\GRiemann$ satisfies
$$
e^{\mub \frac{L}{2 \lambda_{\min} }} \big\lVert \GRiemann  \big\rVert_2\leq 1 
\quad\text{for}\quad
\lambda_{\min}
\coloneqq
\min_{\stackrel{q=1,\ldots,Q,}{x \in [0,L]}}
\bigg\{ 
\GK\big[\lambdasteady^+\big]\big(x;\xi^{(q)}\big),\, 
\Big| \GK\big[\lambdasteady^-\big]\big(x;\xi^{(q)}\big)
\Big| \bigg\},
$$
where~$\xi^{(1)}, \ldots, \xi^{(Q)} $ denote~${Q\coloneqq \big\lceil \frac{3}{2}(K+1) \big\rceil}$  
Gaussian quadrature nodes.

\end{corollary}

\begin{proof}
Since the orthogonal matrices~$\boldsymbol{\widehat{\mathcal{T}}^\pm}$ preserve the spectral matrix norm, we have~${\big\lVert\Ghat\big\rVert_2=\big\lVert\GRiemann\big\rVert_2}$. 
Due to the relation
\begin{equation}\label{LabelCorollary}
{
\boldsymbol{\widehat{\mathcal{T}}^\pm}(x) \boldsymbol{\widehat{\mathcal{D}}^\pm}(x) \boldsymbol{\widehat{\mathcal{T}}^\pm}(x)^{\T}
=
\boldsymbol{\widehat{A}^\pm}(x)
=
\bigg(
\Big\langle
 \GK\big[\lambdasteady^\pm\big](x;\cdot),
 \phi_{\boldsymbol{i}}(\cdot) \phi_{\boldsymbol{j}}(\cdot)
 \Big\rangle_{\mathbb{P}} \bigg)_{\boldsymbol{i},\boldsymbol{j}\in\K}
}
\end{equation}
the eigenvalues~$\boldsymbol{\widehat{\mathcal{D}}^\pm}$ coincide with those of the matrix containing the projected random eigenvalues. These are in turn within the range of the random spectrum of the random characteristic speeds~\cite[Th.~2]{Sonday2011}. The integrals in equation~\eqref{LabelCorollary} are computed exactly with Gaussian  quadrature~nodes~$\xi^{(q)}$ and weights~$\widehat{w}^{(q)}$ as
\begin{align*}
\Big\langle
\GK\big[\lambdasteady^\pm\big](x;\cdot),
\phi_{\boldsymbol{i}}(\cdot) \phi_{\boldsymbol{j}}(\cdot)
\Big\rangle_{\mathbb{P}} 
&=
\sum\limits_{\boldsymbol{k} \in \K }
\prod_{\ell=1}^M
{\widehat{\lambda}^\pm_{k_\ell}}(x) \quad\quad
\Big\langle
\phi_{k_\ell}(\xi_\ell),
\phi_{i_\ell}(\xi_\ell)
\phi_{j_\ell}(\xi_\ell)
\Big\rangle_{\mathbb{P}} \\
&=
\sum\limits_{\boldsymbol{k} \in \K }
\prod_{\ell=1}^M
{\widehat{\lambda}^\pm_{k_\ell}}(x) 
\Bigg[
\sum\limits_{q=1}^Q
\phi_{k_\ell}(\xi^{(q)})
\phi_{i_\ell}(\xi^{(q)})
\phi_{j_\ell}(\xi^{(q)})
\widehat{w}^{(q)} \Bigg]. 
\end{align*}
Hence, the random characteristic speeds can be estimated by evaluating only at Gaussian quadrature nodes.

\end{proof}

\noindent
Although the statement of Theorem~\ref{TheoremStability} seems to be abstract, it is in fact quite practical. 
All quantities can be computed numerically. 
The control only influences the dissipativity condition~\eqref{ThInequality2}. It is a modification of the well-known deterministic dissipativity condition
\begin{equation}\label{Dissipativity}
\rho_2\big( B \big)
\coloneqq
\inf\limits_{ \mathcal{D}\in\mathbb{D} }
\Big\{
\big\lVert
\mathcal{D} 
B
\mathcal{D}^{-1}
\big\rVert_2
\Big\}
<1,
\end{equation}
where $\mathbb{D}$ denotes the set of strictly positive diagonal matrices~\cite[Sec.~2]{L7}.  
Corollary~\ref{CorDissipativityCond}  allows for a practical verification of this condition.



\section{Nonlinear hyperbolic conservation laws}\label{SectionNonlinear}

The Lyapunov stability analysis in Section~\ref{Sec4} is based on  $L^2$-solutions~\cite[Appendix~A]{O1} to~\emph{linear} hyperbolic balance laws. Extensions to \emph{nonlinear} systems are only partial, since 
Lyapunov’s indirect method~\cite{Khalil}  does
not necessarily hold for hyperbolic systems, i.e.~the exponential stability of a nonlinear hyperbolic system does not follow from the stability of its linearization~\cite[Sec.~4]{O1}.  
Furthermore, 
solutions to conservation laws exist
in the classical sense only in finite time due to the occurence of shocks~\cite{Riemann1859,Dafermos}. Hence, the stochastic Galerkin formulation should be performed for a conservative formulation. There are several attempts to deduce hyperbolic stochastic Galerkin formulations for nonlinear hyperbolic systems. We mention here briefly  three approaches for 
shallow water equations~\cite{dai2020hyperbolicity}, where the gPC modes~$\hat{h}$ and $\hat{q}$ describe random water height and mass flux, as well as for  isothermal Euler equations~\cite{GersterJCP2019,S5} with unknown density~$\hat{\rho}$. The parameters are the gravitational constant $g>0$ and the speed of sound $a>0$. Furthermore, the {entropy closure}~\cite{H1} yields a hyperbolic stochastic Galerkin formulation for general systems that are endowed with a strictly convex entropy. \\

\textbf{Shallow water equations:}\vspace*{2mm}

$\displaystyle
\begin{aligned}
&\partial_t
\begin{pmatrix}
\hat{\rho}(t,x) \\
\hat{q}(t,x) 
\end{pmatrix}
+
\begin{pmatrix}
\hat{q}(t,x) \\
\mathcal{P}\big(\hat{q}(t,x)\big) \mathcal{P}^{-1}\big(\hat{\rho}(t,x)\big)
\hat{\rho}(t,x)
+
\frac{g}{2} \mathcal{P}\big(\hat{\rho}(t,x)\big) \hat{\rho}(t,x)
\end{pmatrix}
=0 \\
&\text{for}\quad
\mathcal{P}(\hat{\rho})
\coloneqq
\sum\limits_{ \boldsymbol{k} \in \K }
\hat{\rho}_{\boldsymbol{k}}
\Big(
\big\langle
\phi_{\boldsymbol{k}}(\cdot),
\phi_{\boldsymbol{i}}(\cdot) \phi_{\boldsymbol{j}}(\cdot)
\big\rangle_{\mathbb{P}}
\Big)_{\boldsymbol{i},\boldsymbol{j} \in \K}
\end{aligned}
$

\vspace{4mm}
	
\textbf{Roe variable transform for isothermal Euler equations:}\vspace*{2mm}

$\displaystyle
\begin{aligned}
&\partial_t
\begin{pmatrix}
	\hat{\rho}(t,x) \\
	\hat{q}(t,x) 
\end{pmatrix}
+
\begin{pmatrix}
\hat{q}(t,x) \\
\mathcal{P}\big(\hat{\beta}(t,x)\big) \hat{\beta}(t,x)
+
a^2 \hat{\rho}(t,x)
\end{pmatrix}
=0\\
&\text{for}\quad \hat{\beta}\coloneqq \mathcal{P}^{-1}(\hat{\alpha}) \hat{q} \quad \text{and} \quad
\hat{\alpha}\coloneqq 
\argmin_{\tilde{\alpha}}
\bigg\{ \frac{\tilde{\alpha} \mathcal{P}(\tilde{\alpha}) \tilde{\alpha} }{3} - \tilde{\alpha}^\T \hat{\rho} \bigg\}
\end{aligned}
$

\vspace{4mm}

\textbf{Entropy Closure:}\vspace*{2mm}

$\displaystyle
\begin{aligned}
&\partial_t 
\Big\langle
y(t,x;\cdot), \phi_{\boldsymbol{k}}(\cdot)
\Big\rangle_{\mathbb{P}} 
+
\partial_x
\Big\langle
f\big( y(t,x;\cdot) \big),\phi_{\boldsymbol{k}}(\cdot)
\Big\rangle
=0 \\
&\text{for}\quad \boldsymbol{k} \in \K \quad \text{and} \quad
y(t,x;\xi) \coloneqq \argmin_{\tilde{y}} 
\bigg\{
\int \eta (\tilde{y}) \d \mathbb{P}
\bigg\}
\quad\text{such that}\quad
\Big\langle
\tilde{y}, \phi_{\boldsymbol{k}}(\cdot)
\Big\rangle_{\mathbb{P}} 
=
\hat{y}_{\boldsymbol{k}}(t,x)
\end{aligned}
$ 

\vspace{4mm}

\noindent
All of these approaches yield for general gPC bases a hyperbolic system in conservative form
\begin{equation}\label{nonlinear}
\partial_t 
\hat{y}(t,x)
+
\partial_x
\hat{f}\big(
\hat{y}(t,x)
\big)
=0
\quad
\text{with diagonalizable Jacobian}
\quad
\D_{\hat{y}}
\hat{f}(\hat{y}) 
=
\boldsymbol{\widehat{\mathcal{T}}}(\hat{y}) 
\boldsymbol{\widehat{\mathcal{D}}}(\hat{y}) \boldsymbol{\widehat{\mathcal{T}}}^{-1}(\hat{y}). 
\end{equation}
Without loss of generality, ordered  
eigenvalues~$
\widehat{\Lambda}(\hat{y})
=
\diag\big\{ \widehat{\Lambda}^+(\hat{y}),\widehat{\Lambda}^-(\hat{y}) \big\}
$ 
with 
$
\widehat{\Lambda}^-(\hat{y})
<0<
\widehat{\Lambda}^+(\hat{y}) 
$ are assumed. We introduce the notation 
$
\hat{y}
=
\big(
\hat{y}^+, 
\hat{y}^-
\big)^\T
$ 
with cardinality 
$
\big| \widehat{\Lambda}^\pm(\hat{y}) \big|
=
\big| \hat{y}^\pm \big|
$. Then, the conservation law~\eqref{nonlinear} can be endowed with the  possibly nonlinear boundary conditions
\begin{equation}\label{nonlinearBC}
\begin{pmatrix}
\hat{y}^+(t,0) \\
\hat{y}^-(t,L) 
\end{pmatrix}
=
\Gphysicsnonlinear
\begin{pmatrix}
\hat{y}^+(t,L) \\
\hat{y}^-(t,0) 
\end{pmatrix}
 \quad\text{for}\quad
 \Gphysicsnonlinear \in C^2\Big( \mathbb{R}^{2|\K|}, \mathbb{R}^{2|\K|} \Big)
\end{equation}
such that the boundary value  problem~\eqref{nonlinear}~and~\eqref{nonlinearBC} is well-posed~\cite[Appendix~B]{O1}. 
According to \cite[Sec.~1~and~Sec.~4]{O1}, 
a direct extension of the $L^2$-stability analysis in Section~\ref{Sec4} would require a transform in Riemann coordinates, i.e.~a diffeomorphism~$
\varphi(\hat{y})=\Rd
$ 
satisfying
\begin{equation}\label{ChangeCoordinates}
\D_{\hat{y}} \varphi(\hat{y})
\D_{\hat{y}} \hat{f}(\hat{y})
=
\boldsymbol{\widehat{\mathcal{D}}}(\hat{y})
\D_{\hat{y}} \varphi(\hat{y}).
\end{equation}
Then, the system~\eqref{nonlinear} is for $C^1$-solutions with norm 
$$
\Big\lVert \hat{y}(t,\cdot) \Big\rVert_{C^1}
\coloneqq
\max\limits_{x\in[0,L]}
\Big\{
\big\lVert
\hat{y}(t,x)
\big\rVert_0
+
\big\lVert
\partial_x
\hat{y}(t,x)
\big\rVert_0
\Big\} <\infty
\quad\text{and}\quad
\big\lVert
\hat{y}(t,x)
\big\rVert_0
\coloneqq
\max\limits_{\boldsymbol{k}\in\K} 
\Big\{\big|\hat{y}_{\boldsymbol{k}}(t,x)
	\big|\Big\}
$$
equivalent to the diagonal form 
$
\partial_t \Rd 
+
\boldsymbol{\widehat{\mathcal{D}}} \big( \varphi^{-1}(\Rd) \big) \partial_x \Rd=0. 
$ 
The transform~\eqref{RiemannInvariants} for linear systems is a special case, i.e.~$\varphi(\hat{y})=T^{-1}\hat{y}$. 
		For general nonlinear systems, however, finding the change of coordinates~$\varphi$ requires to solve the first-order differential equation~\eqref{ChangeCoordinates}. 
For systems of size~$|\hat{y}|\geq 3$  the solution~$\varphi$ exists~\emph{only in specific cases}~\cite{Lax1973,Tsien1997,O1}. 
This issue can be partially circumvented for $C^1$-solutions to conservation laws of the  form~\eqref{nonlinear} with constant steady states, where at least the Jacobian is diagonalizable. The following theorem states how boundary conditions can  be specified such that $C^1$-solutions are stable.\\
 
\begin{theorem}[{According to~\cite[Def.~4.17 and~Th.~4.18]{O1}}] \label{TheoremNonLinear}
Assume boundary conditions~\eqref{nonlinearBC} of the form
$$
\rho_\infty\big(\Gphysicsnonlinear'\big)
\coloneqq
\inf\limits_{ \mathcal{D}\in\mathbb{D} }
\Big\{
\big\lVert
\mathcal{D} 
\Gphysicsnonlinear'
\mathcal{D}^{-1}
\big\rVert_\infty
\Big\}
<1
\quad\text{with Jacobian}\quad
 {\Gphysicsnonlinear'}
 \coloneqq
 \D_{\hat{y}} \Gphysicsnonlinear(\hat{y})\big|_{\hat{y}=\hat{y}^*}.
$$
Then, there are  constants~$\mu,\varepsilon>0$ such that
\begin{align*}
&\big\lVert \hat{y}(t,\cdot)
-
\hat{y}^*(\cdot)
\big\rVert_{C^1}
\lesssim
e^{-\mu t} 
\big\lVert \hat{y}(0,\cdot)
-
\hat{y}^*(\cdot)
\big\rVert_{C^1} \\
&\text{for all initial values}\quad
\hat{y}(0,\cdot) \in \bigg\{
\hat{y}_0 \in C^1\Big( [0,L]; \mathbb{R}^{2|\K|}\Big) \ \Big| \ 
\big\lVert \hat{y}_0
-
\hat{y}^*
\big\rVert_{C^1} < \varepsilon
\bigg\}.
\end{align*}


\end{theorem}


\noindent
Theorem~\ref{TheoremNonLinear} also allows to couple systems on different arcs, as illustrated in Figure~\ref{FigNetwork3}.

\vspace{1cm}
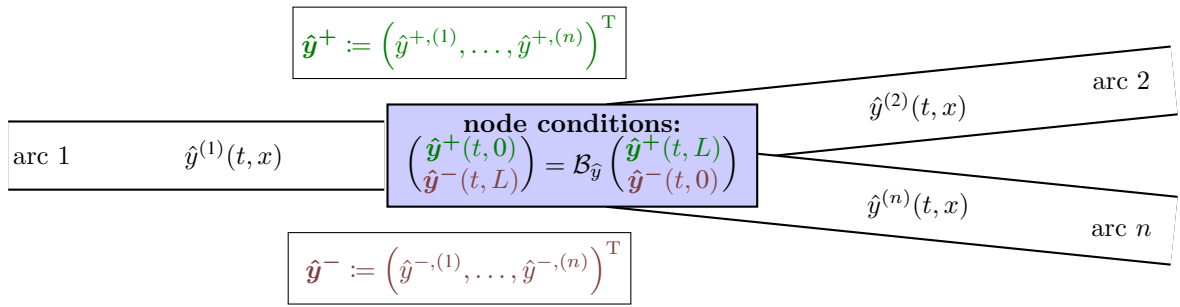
\begin{figure}[h]
	\tikzstyle{block} = [thick,draw, fill=blue!20, rectangle, 
minimum height=3em, minimum width=3em]
\tikzstyle{Kreis} = [thick,draw, fill=blue!20, circle, 
minimum height=3em, minimum width=3em]
\tikzstyle{sum} = [draw, fill=darkgreen!20, rectangle, 
minimum height=2em, minimum width=3em]
\tikzstyle{input} = [coordinate]
\tikzstyle{output} = [coordinate]
\tikzstyle{pinstyle} = [ pin edge={to-,thick,black}]
\tikzstyle{vecArrow} = [thick,double distance=25pt]

\vspace{-1cm}\hspace{-5mm}
\begin{tikzpicture}[scale=0.1]

\draw [vecArrow] (0,0) -- (50,0);
\draw [vecArrow] (155,10) -- (80,2.3);
\draw [vecArrow] (155,-10) -- (80,-2.3);

\node [block,align=center] at (75,0) {\textbf{node conditions:}\\
\	$
\begin{pmatrix}{\color{gruen} \boldsymbol{\hat{y}^+} (t,0) }\\ {\color{braun}\boldsymbol{\hat{y}^-}(t,L)} \end{pmatrix} 
=
\Gphysicsnonlinear
\begin{pmatrix}\color{gruen}  \boldsymbol{\hat{y}^+}(t,L) \\ \color{braun}\boldsymbol{\hat{y}^-}(t,0) \end{pmatrix} 
$ \
  };

\node at (30,0) {$\hat{y}^{(1)}(t,x) $};
\node at (121,6.5) {$\hat{y}^{(2)}(t,x) $};
\node at (121,-6.5) {$\hat{y}^{(n)}(t,x) $};

\node at (4.5,0) {arc 1};
\node at (148,10) {arc 2};
\node at (148,-10) {arc $n$};

\node at (60,15) {
$ \boxed{\color{gruen} \boldsymbol{\hat{y}^+}\coloneqq \Big( \hat{y}^{+,(1)},\ldots,\hat{y}^{+,(n)} \Big)^\T}$
};

\node at (60,-15) {
$ \boxed{\color{braun}\ \boldsymbol{\hat{y}^-}\coloneqq \Big( \hat{y}^{-,(1)},\ldots,\hat{y}^{-,(n)} \Big)^\T}$
};







\end{tikzpicture}
	\caption{Network with $n$ arcs for  dynamics described by the boundary value  problem~\eqref{nonlinear}~and~\eqref{nonlinearBC}.}
	\label{FigNetwork3}
\end{figure}

\noindent 
However in comparison to the stability analysis for linear systems with dissipativity condition~\eqref{Dissipativity}, Theorem~\ref{TheoremNonLinear} does not account for source terms and  the $C^1$-solution concept requires more restrictive assumptions on the boundary conditions~\cite[Prop.~4.7]{O1}, since the inequality~$\rho_2\big(\Gphysicsnonlinear'\big) \leq \rho_\infty\big(\Gphysicsnonlinear'\big)$ holds. 
Furthermore, Theorem~\ref{TheoremNonLinear} 
 requires initial data to be close to a steady state. This makes  well-balanced numerical schemes essential. 
We mention here briefly that well-balanced schemes have been developed for stochastic Galerkin formulations to scalar equations~\cite{H5} and to shallow water equations~\cite{dai2020hyperbolicity}. Since they must be specifically adopted to a particular model, we  limit ourselves in the following to linear systems.

\newpage
\section{Elastic and viscoplastic deformations}\label{SectionFeedback}

Elastic deformations are typically small displacements about a reference state~\cite[Sec.~2.12.1]{Leveque}.  The displacements lead to \textbf{strains},  i.e.~small deformations within the material. 
The interior forces because of stretching and compression of atomic bonds are called~\textbf{stress}. These forces result in acceleration of the material and in turn affect the evolution of strains. 
Hence, the constitutive equations consist of Newton's law, relating stress and force to acceleration, and a~\textbf{stress-strain relationship}.
Viscoplastic deformations can be modelled by decomposing the  total strain
$${\epsilon(t,x)=\epsilon^e(t,x)+\epsilon^p(t,x)}$$
into an elastic part~$\epsilon^e$ and plastic part~$\epsilon^p$. 
There are several one-dimensional models  that describe  a  displacement $u(t,x)$ and   stress~$\sigma(t,x)$ at time~${t\geq 0}$ and position~${x\in[0,L]}$.
An overview on various models can be found e.g.~in~\cite{Inelasticity}. 
Here, the elastic relationship is described by
$${
\sigma(t,x)
=
E\big(
\epsilon(t,x) - \epsilon^p(t,x)
\big)
}$$
with a constant~${E>0}$. 
A given stress-strain relationship allows to view the plastic part as a function of the stress, i.e.~$\bar{\epsilon}^p(\sigma)$. 
Figure~\ref{measurements} shows two models and typical measurements  in the left panel.  A Bergstr\"om-model 
states a stress-strain relationship in terms of the ordinary differential equation
$$
\frac{\partial\rho(\epsilon)}{\partial \epsilon} 
=
\frac{\dot{\epsilon} U_0}{T}
\sqrt{\rho(\epsilon)}
-\Bigg( \Omega_0+C\text{exp}  \Big(
-\frac{mQ}{RT}
\dot{\epsilon}^{-m} 
\Big) 
\Bigg)
\rho(\epsilon),
\ \
\sigma(\epsilon)
=
\sigma_0+\alpha G b \sqrt{ \rho(\epsilon) }.
$$
We refer the reader to~\cite{BERGSTROM}, where a description of typical model parameters is given. This model is extended by the DRX-model that takes a softening due to dynamically recrystallized grains into account. 
If a critical stress~${\epsilon_c>0}$ is reached, the Bergstr\"om-model is replaced by
\begin{align*}
X(\varepsilon)
&\coloneqq
1-
\textup{exp}\left(
-\kappa
\Big(
\frac{\epsilon-\epsilon_c}{\epsilon_s-\epsilon_c}
\Big)^{q}
\right),\quad
X_\sigma(\epsilon) 
\coloneqq 
\int_{\epsilon_c}^{\epsilon}
X'(\epsilon) \sigma(\epsilon-\epsilon_c) \d\epsilon, \\
\sigma(\epsilon) &\coloneqq \sigma(\epsilon_c) \big( 1-X(\epsilon) \big)+X_\sigma(\epsilon).
\end{align*}
Here, $X(\epsilon)$ is the recrystallized volume fraction and $\kappa,q>0$ are further parameters. We refer the reader to~\cite{Bambach2013} for a more detailed description.

\begin{figure}[h]
\begin{center}
\begin{minipage}{0.6\linewidth}
	\scalebox{1}{\includegraphics[width=\linewidth]{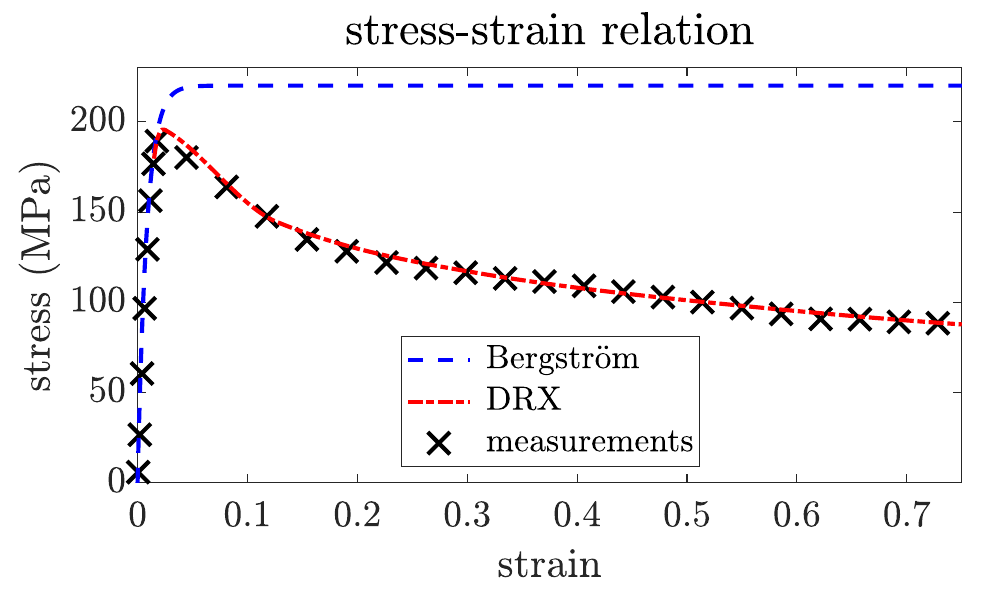}}
\end{minipage}
\qquad\qquad
\begin{minipage}{0.18\linewidth}
	\scalebox{0.9}{\includegraphics[width=\linewidth]{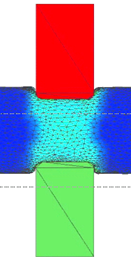}}
\end{minipage}
	\caption{Left: Bergstr\"om- and DRX-model compared to sample measurements.  
		Right:  
Sketch of a forming process.
}
	\label{measurements}
\end{center}
\end{figure}

\noindent
Then, the displacement velocity, i.e.~the acceleration~${v(t,x)\coloneqq \partial_t u(t,x)}$, and the stress are described by the balance law
$$
\frac{\partial}{\partial t} 
\begin{pmatrix} v(t,x) \\ \sigma(t,x) \end{pmatrix}
+
\begin{pmatrix} & -1 \\ -E \\ \end{pmatrix}
\frac{\partial}{\partial x} 
\begin{pmatrix} v(t,x) \\ \sigma(t,x) \end{pmatrix}
=
\begin{pmatrix} 0 \\ \bar{\epsilon}^p \big(\sigma(t,x) \big) \end{pmatrix}.
$$
This model is used to describe a one-dimensional cross-section of a forming process as illustrated in the right panel of~Figure~\ref{measurements}.

\subsection{Derivation of a control rule}\label{SubSecControlRule}

Microstructural features such as grain size 
	can be estimated from the stress. This is necessary to impose plastic deformations. 
	The stress can be used as indicator for the microstructural states which cannot be measured during a forming process. 
	Hence, there is an interest to control  the system to a desired stress~$\sigma^*$. 
	We assume a source term that is linear close to the desired state, i.e.~$
\bar{\epsilon}^p (\sigma )
=
\bar{\epsilon}^p (\sigma^* )
+
\DerivStress 
\Delta \sigma
$
with a constant~$\DerivStress\coloneqq \partial_\sigma \bar{\epsilon}^p(\sigma^*) \in\mathbb{R}$ that is obtained by a Taylor series approximation. Then, 
the deviations of the desired state are described by the linear system
\begin{align*}
&\partial_t \Delta y(t,x) + A \partial_x \Delta y(t,x) = -S(x) \Delta y(t,x) \\
&\text{with} \quad
A = \begin{pmatrix} & -1 \\ -E \\ \end{pmatrix}
\quad\text{and}\quad
S = -\begin{pmatrix} 0 & 0 \\ 0 & \DerivStress	\end{pmatrix}.
\end{align*}
In this particular example the advection part is uniform, but the source term is space-varying. 
We diagonalize  the Jacobian
$$
A = \Tsteady \Lambdasteady \Tsteady^{-1}
\quad \text{with} \quad 
\Tsteady = \begin{pmatrix}
-1 & 1 \\ \sqrt{E} & \sqrt{E}
\end{pmatrix}
\quad \text{and} \quad 
\Lambdasteady = \begin{pmatrix}
\sqrt{E} \\ & -\sqrt{E}
\end{pmatrix}
$$
to obtain the Riemann invariants and the corresponding source term as
\begin{equation*}
\Riemann 
=
\begin{pmatrix} \Riemann^+ \\ \Riemann^- \end{pmatrix}
\coloneqq \Tsteady^{-1} \Delta y
=
\frac{1}{2}
\begin{pmatrix}
\frac{\Delta \sigma}{\sqrt{E}} - \Delta v \\
\frac{\Delta \sigma}{\sqrt{E}} + \Delta v
\end{pmatrix}, \quad
\SRiemann = \Tsteady^{-1}S\Tsteady = - \frac{\DerivStress}{2}
\begin{pmatrix}
1 & 1 \\ 1 & 1
\end{pmatrix}.
\end{equation*}
For the system~${
	\Riemann_t + \Lambdasteady \Riemann_x = -\SRiemann \Riemann
}$ we prescribe the linear feedback boundary conditions
\begin{equation*} 
\begin{pmatrix} \Riemann^+(t,0) \\ \Riemann^-(t,L) \end{pmatrix}
=
\begin{pmatrix} & \kappa_0 \\ \kappa_1  \end{pmatrix}
\begin{pmatrix} \Riemann^+(t,L) \\ \Riemann^-(t,0) \end{pmatrix}.
\end{equation*}


\noindent
These boundary conditions, reformulated in terms of the displacement velocity, read as 
\begin{equation} \label{feedbacklaw}
\begin{aligned}
v(t,0)
&=
v^*(0)
+
\frac{1 - \kappa_0 }{ \sqrt{E}+ \kappa_0 \sqrt{E} }
\Big(
\sigma(t,0)-\sigma^*(0)
\Big), \\
v(t,L)
&=
v^*(L)
+
\frac{\kappa_1-1 }{ \sqrt{E}+\sqrt{E}\kappa_1 }
\Big(
\sigma(t,L)-\sigma^*(L)
\Big),
\end{aligned}
\qquad \text{i.e.~the  control }
\Gphysicslinear=
\begin{pmatrix}
\frac{1 - \kappa_0 }{ \sqrt{E}+ \kappa_0 \sqrt{E} } &  \\
& \frac{\kappa_1-1 }{ \sqrt{E}+\sqrt{E}\kappa_1 } 
\end{pmatrix}
\end{equation}
for the linear boundary conditions~\eqref{BC} is obtained. 
Parameters~${\kappa_0,\kappa_1 \in \mathbb{R}}$ should be specified such that the dissipativity condition~\eqref{ThInequality2} in Theorem~\ref{TheoremStability} holds. Furthermore, a large constant~$\mub>0$ is desirable, which leads to a larger decay rate~$\mu$. For instance, the ansatz 
$
\mu(\mub)
\geq 0
$, 
$
\textup{exp}\big(\mub \frac{L}{2 \lambda_{\min} }\big) \big\lVert\GRiemann \big\lVert= 1 
$ 
yields
\begin{align*}
\mu(\mub)
&=
\mub
-
\bigg|
\frac{\DerivStress}{2}
\bigg|
\max\limits_{x\in[0,L]}
\bigg\{
	3 e^{-\frac{\mub x}{\sqrt{E}}}
	+
	e^{-\frac{\mub (L-x)}{\sqrt{E}}}, \,
	e^{-\frac{\mub x}{\sqrt{E}}}
+
3e^{-\frac{\mub (L-x)}{\sqrt{E}}}
\bigg\} \\
&\geq
\mub - 2\big| \DerivStress \big|=0
\quad\text{for}\quad
\mub \coloneqq 2\big| \DerivStress \big|.
\end{align*}
With this choice, the dissipativity condition is satisfied for 
${\kappa_i = \kappa \coloneqq
e^{
\frac{L}{\sqrt{E}}
| \DerivStress |	
}}$ and the feedback law~\eqref{feedbacklaw} reads as
$$
v(t,0)
=
v^*(0)
+
\frac{1}{\sqrt{E}}
\coth\bigg(
\frac{L}{\sqrt{E}}
\big| \DerivStress \big|
\bigg)
\Big(
\sigma(t,0)-\sigma^*(0)
\Big).
$$


\subsection{Gaussian random fields}\label{Sec5}

Uncertainties due to unknown model parameters and uncertain steady states can be described by Gaussian random fields. These are completely characterized by their mean and covariance kernel. To be precise, for any given~${n\in\mathbb{N}}$, all locations~${\vec{x}\coloneqq(x_{1},\ldots,x_{n})^\T \in\mathbb{R}^n}$ and all realizations   ${\vec{z}\coloneqq(z_{1},\ldots,z_{n})^\T \in\mathbb{R}^n}$ the probability distribution of the event 
$ \big\lbrace 
\omega\in\Omega \ \big| \ 
\process(x_1;\omega)=z_{1}, \ldots, \process(x_n;\omega)=z_{n} \big\rbrace$ is given by the multivariate Gaussian probability density
\begin{align}
&f_{(\mu,\Sigma)}( \vec{z} ) = 
\frac{1}{\sqrt{
		(2\pi)^n
		 \textup{det}\big(\Sigma(\vec{x})\big) 
}}
\exp\Big(- \frac{1}{2} \big(\vec{z}-\mu(\vec{x}) \big)^\T \Sigma(\vec{x})^{-1} \big(\vec{z}-\mu(\vec{x})\big) \Big) \nonumber \\
& \text{with} \quad
\mu(\vec{x}) \coloneqq
\Big(
\mathbb{E}\big[
\process(x_i;\omega)
\big]\Big)_{i=1,\ldots,n}
\quad\text{and}\quad
\Sigma(\vec{x}) \coloneqq \Big( \C (x_i,x_j) \Big)_{i,j=1,\ldots,n}.\label{MeanVarianceGaussian}
\end{align}

\noindent
Furthermore, measurements can be taken into account. 
Intuitively, one may think of simulating from a Gaussian distribution and then rejecting all sample paths that disagree with the measurements. Note, however, that a Gaussian process restricted to measurements is again a Gaussian process. 
We assume measurements~${\vec{z}^*=(z^*_1,\ldots,z^*_m)^\T\in\mathbb{R}^m}$ at the locations~${\vec{x}^*=(x^*_1,\ldots,x^*_m)^\T}$. Then, the conditioned process satisfies	
$$
\process^*(x_\ell^*;\omega)
\coloneqq
\process(x_\ell^*;\omega)\big|_{\process(x^*)=z^* }
=
z^*_\ell
\quad
\text{for all}
\quad
\ell=1,\ldots,m. 
$$
According to~\cite[A.2]{Rasmussen}, the mean and the covariance matrix of the conditioned process~$\process^*$ read as 
\begin{align*}
\mu^*(\vec{x})
\coloneqq&
\Big(
\mathbb{E}\big[
\process^*(x_i;\omega)
\big]
\Big)_{i=1,\ldots,n} \\
=&
\mu(\vec{x})+
\big(\C(x_i,x^*_\ell) \big)_{\stackrel{i=1,\ldots,n,}{\ell=1,\ldots,m}}
\big(\C(x_k^*,x_\ell^*)\big)_{k,\ell=1,\ldots,m}^{-1}
z^*, \\
\Sigma^*(\vec{x})
\coloneqq&
\Big(
\textup{Cov}
\big[
\process^*(x_i;\omega),
\process^*(x_j;\omega)
\big]\Big)_{i,j=1,\ldots,n} \\
=&\big(\C(x_i,x_j)\big)_{i,j=1,\ldots,n} 
-
\big(\C(x_i,x^*_k) \big)_{\stackrel{i=1,\ldots,n,}{k=1,\ldots,m}}
\big(\C(x_k^*,x_\ell^*)\big)_{k,\ell=1,\ldots,m}^{-1}
\big(\C(x_k^*,x_j) \big)_{\stackrel{k=1,\ldots,m,}{j=1,\ldots,n}}
\end{align*}
where~$\mu$ and~$\C$ are the mean and covariance kernel of the unconditioned process~\eqref{MeanVarianceGaussian}. 
A widely used class of covariance kernels that generate  strictly positive definite matrices~$\Sigma(\vec{x})$ is the Mat\'ern covariance function 
$$  \C_\nu (x_i,x_j) \coloneqq \sigma^2
\frac{2^{1-\nu}}{
	\Gamma(\nu)
}
\bigg( \sqrt{2 \nu} \frac{|x_i - x_j|}{\lambdaMatern} \bigg)^\nu K_v \bigg(  \sqrt{ 2  \nu} \frac{|x_i - x_j|}{\lambdaMatern}  \bigg).
$$
Here, $\Gamma$ is the gamma function,  $K_\nu$ is the modified Bessel function of the second kind, $\sigma^2>0$ denotes the variance and~${\lambdaMatern>0}$ is a scaling parameter, which describes the spatial correlation of the stochastic process. 
The sample paths are $ \nu - 1$ times differentiable~\cite{Rasmussen}. 
Specifically in  Figure~\ref{Kriging}, we consider the 
\begin{align*}
&\text{exponential kernel}
&\C_{\nicefrac{1}{2}}(x,y)
&=
\sigma^2\exp\bigg(
-\frac{|x-y|}{\lambdaMatern}
\bigg),\\
&\text{Gaussian kernel}
&\lim\limits_{\nu\rightarrow\infty}\C_{\nu}(x,y)
&=
\sigma^2\exp\bigg(
-\frac{1}{2}
\Big(
\frac{x-y}{\lambdaMatern}
\Big)^2
\bigg).
\end{align*}
The exponential covariance kernel describes a stationary Ornstein-Uhlenbeck process, where sample paths are not differentiable.  
The Gaussian kernel yields smooth sample paths. 
The upper panels of Figure~\ref{Kriging} show four  simulations of stress, where an exponential (left) and Gaussian (right) covariance structure is assumed as a priori distribution. The mean is plotted as dashed, black line. The $0.95$ confidence region~$\textup{CI}(x)$, satisfying 
$
\mathbb{P}\big[\process(x;\omega) 
\in 
\textup{CI}(x)
\big]
=0.95,
$ 
is shown in grey. The lower panels include measurements at the boundaries~$x_1^*=0$ and~$x_2^*=L$. 
As explained, the posteriori distribution that accounts for measurements is again a Gaussian distribution defined by~$\mu^*$ and~$\Sigma^*$. We observe that uncertainties are substantially decreased close to the measurements, but still mostly described by the a priori distribution in the interior of the spatial domain. \\

\begin{figure}[h]
	\scalebox{1}{\includegraphics[width=\linewidth]{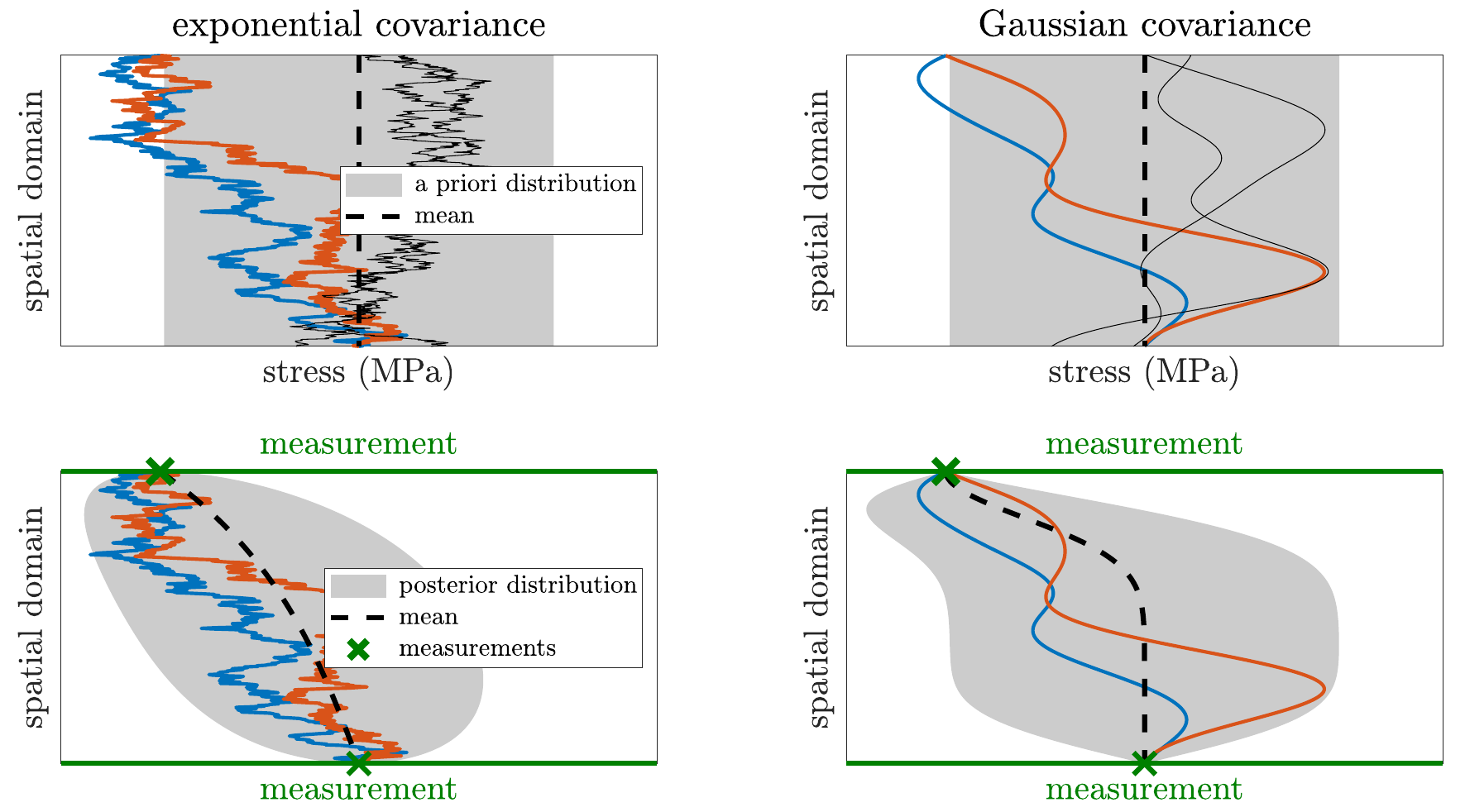}}
	\caption{Upper panels show the a priori Gaussian distribution of an uncertain stress.   Lower panels account for measurements~(green) and show the posteriori distribution.}
	\label{Kriging}
\end{figure}

\noindent
The eigenfunctions and eigenvalues of the strictly positive definite kernel~$\C$ are given by the \textbf{Fredholm integral equation}
\begin{equation}\label{Fredholm}\tag{FI}
\int_0^L \C(x,\cdot) \psi_k(x)  \d x
=
\lambdaKL_k \psi_k(\cdot)
\quad\text{with}\quad
k\in\mathbb{N}.
\end{equation}
The Fredholm integral equation~\eqref{Fredholm} has been analyzed intensively. In particular,  explicit solutions are known for the exponential and Gaussian kernels~\cite{Rasmussen}. Furthermore, there are software packages available, for instance the~\texttt{chebfun}-package~\cite{chebfun}, 
that allow for numerical solutions to general kernels.  In any case, the truncation~${\NKL\in\mathbb{N}}$ must be specified. Typical indicators are the explained ratio between the
\begin{equation}\label{explainedVariance}
\text{total variance} \quad
\frac{\sum\limits_{k=1}^{\NKL} \lambdaKL_k }{\int_0^L \C(x,x) \d x}
\quad\text{and pointwise variance} \quad
\frac{\sum\limits_{k=1}^{\NKL} \lambdaKL_k\psi_k^2(x) }{\C(x,x)}.
\end{equation}
Figure~\ref{KarhunenLoeve_Exponential} illustrates these error estimates for the exponential (top) and Gaussian (bottom) covariance kernel. In each case the eigenfunctions~$\psi_1,\ldots,\psi_4$ are  smooth and approximate the stochastic process as superpositions in the series expansion~\eqref{KL}. The eigenvalues of the Gaussian kernel give more weight on the first eigenfunctions and hence, less  eigenfunctions are necessary to approximate smooth sample paths. The explained total  variance~\eqref{explainedVariance} is shown in the last panels with respect to the left vertical axis and the pointwise variance is in the scale shown at the right vertical axis.

\begin{figure}[htp]
	\scalebox{1}{\includegraphics[width=\linewidth]{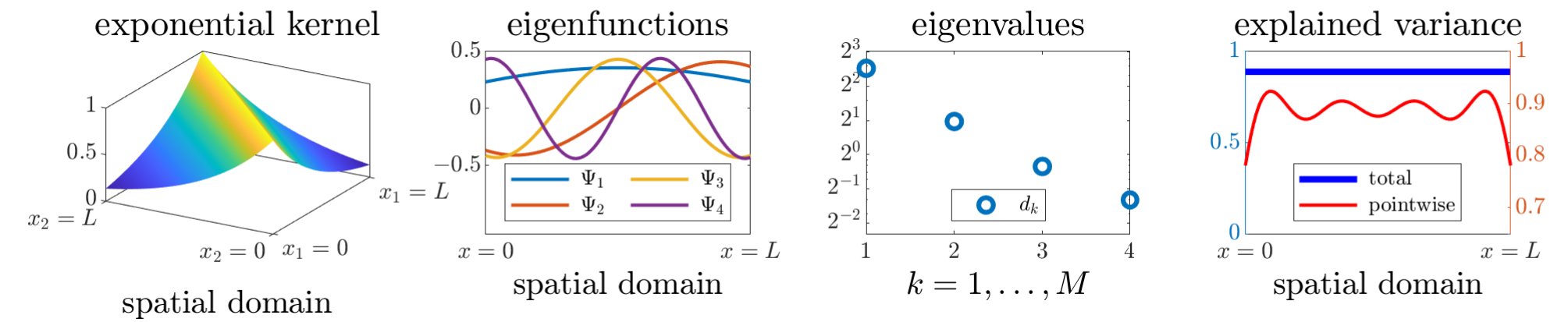}}

\vspace*{4mm}
	\scalebox{1}{\includegraphics[width=\linewidth]{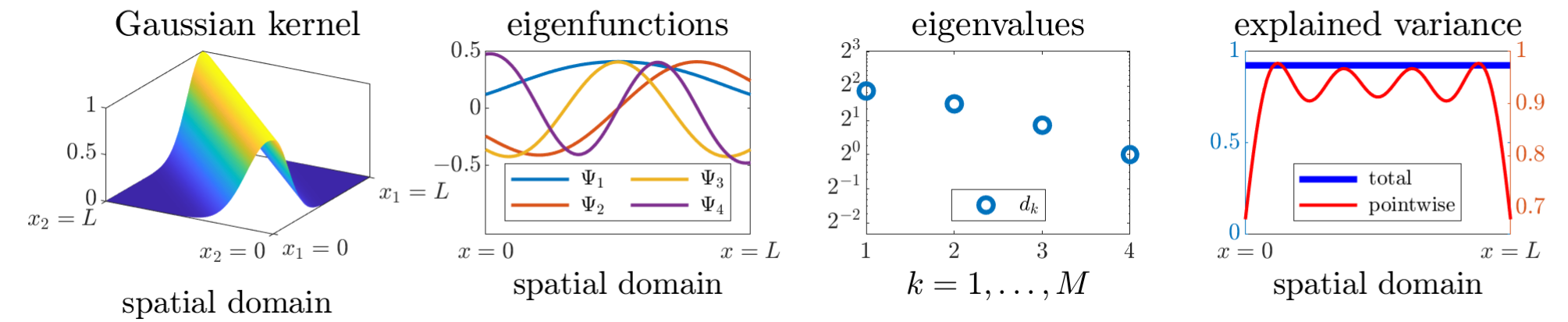}}
	\caption{Karhunen-Lo\`{e}ve decomposition for the exponential kernel~(top) and Gaussian kernel~(bottom).}
	\label{KarhunenLoeve_Exponential}
\end{figure}


\noindent
\begin{remark}
	
	The $L^2$-solution concept~\cite[Th.~A.4]{O1} requires that the characteristic speeds are Lipschitz continuous and the source term must be continuous in~$x\in[0,L]$. 
	Realizations corresponding to the exponential kernel are \emph{not} Lipschitz-continuous. Hence, the  deterministic stability concept cannot be applied if realizations of the characteristic speeds~\eqref{RandomDistinctEW} arise from a random field with exponential covariance structure. 
	In contrast, the stability in mean squared sense may still be considered, since smooth eigenfunctions result in the truncated Karhunen-Lo\`{e}ve expansion in smooth characteristic speeds~$\eqref{DistinctDeterministicEW}$. 
In the following we consider uncertainties in the source term, where this ambiguity does not arise.

Furthermore, the indicators~\eqref{explainedVariance} only state a reasonable truncation~$M\in\mathbb{N}$ for the Karhunen-Lo\`{e}ve decomposition~\eqref{KL}. 
The polynomial chaos expansion~\eqref{NgPC} requires additionally a truncation~$K\in\mathbb{N}$. Finding appropriate choices  is subject of current research. 
For instance, convergence rates have been established for the linear advection equation~\cite{H3}, for a linear BGK~model~\cite{Hui2020} and for scalar nonlinear equations~\cite{H1} provided that the dependency of the solution on the stochastic input is sufficiently smooth.  
However, results are so far only partial and problem dependent. 
In the following, the choice of gPC truncation~$K\in\mathbb{N}$ is not exactly justified.

\end{remark}

\subsection{Computational results}\label{SectionNumerics}

Using an equidistant space discretization~$\Delta x>0$, the space interval~$[0,L]$ is divided into $N$ cells such that ${\Delta x N = L}$ with cell centers ${x_i \coloneqq \big(i-\frac{1}{2}\big) \Delta x}$  for ${i=1,\ldots,N}$ and cell edges ${x_{i+{\nicefrac{1}{2}}} \coloneqq i \Delta x}$ for ${i=0,\ldots,N}$. Additional cells with centers~$x_0$ and~$x_{N+1}$ are added outside the domain. The discrete time steps are denoted by ${t_k\coloneqq k \Delta t}$ for ${k \in \mathbb{N}_0}$ and ${\Delta t>0}$ such that the CFL-condition
\begin{equation*}
\textup{CFL} \coloneqq  
\max\limits_{\stackrel{x \in [0,L],}{k=1,\ldots,|\K|}}
\Big\{ \DhatPlus_{k}(x) ,\big|\DhatMinus_{k}(x)\big| \Big\}  \
\frac{\Delta t}{\Delta x} \leq 1
\end{equation*}
holds. Cell averages at time steps $t_k$ are approximated by
\begin{equation*}
\Rd_i^k\coloneqq
\begin{pmatrix} \Rd_i^{k,+} \\ \Rd_i^{k,-} \end{pmatrix}
\approx \frac{1}{\Delta x}
\int\limits_{x_{i-{\nicefrac{1}{2}}}}^{x_{i+{\nicefrac{1}{2}}}}
\begin{pmatrix} \Rd^+(t_k,x) \\ \Rd^-(t_k,x) \end{pmatrix}
\d x \in \mathbb{R}^{2|\K | }.
\end{equation*}
The advection part can be approximated by a left and right sided upwind-scheme and the reaction part by the explicit Euler method, i.e.
\begin{equation*}
\Rd_i^{k+1} = \Rd_i^{k} - \frac{\Delta t}{\Delta x}
\begin{pmatrix}
\DhatPlus(x_i) \\
& \DhatMinus(x_i)
\end{pmatrix}
\begin{pmatrix}
\Rd_i^{k,+} - \Rd_{i-1}^{k,+}\\
\Rd_{i+1}^{k,-}  -\Rd_{i}^{k,-}
\end{pmatrix}
- \Delta t \Bhat(x_i) \Rd_i^{k}
\quad\text{ with}\quad
\begin{pmatrix}
\Rd^{k,+}_0 \\ \Rd^{k,-}_{N+1}
\end{pmatrix}
=\Ghat
\begin{pmatrix}
\Rd^{k,+}_{N} \\ \Rd^{k,-}_1
\end{pmatrix}.
\end{equation*}
Since the eigenvalues~$\boldsymbol{\widehat{\mathcal{D}}^\pm}$ and the source term $\Bhat$ are precomputed, the computational complexity of the advection part is linear and the complexity of the source term is quadratic, since it involves a matrix-vector multiplication. Hence,  computational cost is relatively low. The drawback is a restriction to a first-order method. 
To illustrate Theorem~\ref{TheoremStability}  we follow the approach in~\cite[Sec.~4]{Gerster2019} and approximate the continuous Lyapunov function~\eqref{Lyapunov} specified by the matrix~\eqref{Weights} as 
\begin{align}
&\L^k \coloneqq \Delta x \sum_{i=1}^{N}  \
\big(\Rd_i^{k}\big)^\T
W_i \Rd_i^{k} 
\quad\text{with  weights}\quad
W_i\coloneqq \diag\Big\{
W_i^+,W_i^-
\Big\}, 
\quad
W_i^\pm\coloneqq \diag\Big\{
w_i^{1,\pm},\ldots,w_i^{|\K|,\pm}
\Big\} \nonumber \\
&\text{for}\quad
w_i^{k,+}
\coloneqq
\frac{h_k^{+}}{\DhatPlus_k(x_i)}
\prod_{\ell = 1}^{i-1}
 \Bigg( 1-\Delta x \frac{\mub}{ \DhatPlus_k(x_\ell) } \Bigg)
\quad\text{and}\quad
w_i^{k,-}
\coloneqq
\frac{h_{k}^-}{\big|\DhatMinus_k(x_i)\big|}
\prod_{\ell = N}^{i+1}
\Bigg( 1+\Delta x \frac{\mub}{ \DhatMinus_k(x_\ell) } \Bigg).
\label{weightsDiscrete}
\end{align}


\begin{remark}
	
The continuous Lyapunov function for a single Riemann invariant with positive constant characteristic speed~$\lambda>0$ is of the form
\begin{equation*}
\L(t) = \int_0^L \Rd(t,x)^2 w(x) \d x
\quad
\text{for}
\quad
w(x)\coloneqq
e^{-\frac{\mub}{\lambda}x }
\end{equation*}
and a straightforward discretized analogue  would be
\begin{equation*}
\L^{k} = 
 \Delta x \sum\limits_{i=1}^N \big(\Rd_i^{k}\big)^2 e^{-\frac{\mub}{\lambda}x_i } . 
\end{equation*}
However,  a blow up in the discretized derivatives occurs in contrast to the continuous case~\cite{D1}.  
The choices~\eqref{weightsDiscrete} circumvent this problem by approximating the continuous derivative
\begin{equation*}
 w'(x)
= -\mub \frac{w(x)}{\lambda}
 \quad \text{as} \quad
\frac{ w_i - w_{i-1} }{\Delta x}
= -\mub
\frac{w_{i-1}}{\lambda},
\end{equation*}
where the one-sided difference quotient takes the direction of characteristic speeds into account. 
Then, the numerical method transfers continuous stability results to the discretized case and guarantees that also discretized Lyapunov functions  decay exponentially fast over time~\cite[Sec.~4]{Gerster2019}.

\end{remark}


\noindent 
Theoretical results are illustrated for a linearized plastic part
\begin{equation}\label{linearizedplasticpart}
 \bar{\epsilon}^p_{\sigma}\Big(\sigma^*\big(x;\xi(\omega)\big)\Big)
\approx
	\mathcal{K}\Big[ \bar{\epsilon}^p_{\sigma}(\sigma^*) \Big](x;\omega)
\end{equation}
that is modelled by a random field with Gaussian covariance structure. 
Since uncertainties arise only from the source term, the transform~\eqref{TrafoCeta} simplifies to 
$
\Rd =
\Rhat
$ and physical quantities are obtained by 
$$
\GK \big[y\big](t,x;\xi)
=
\GK \big[y^*\big](x;\xi)
+\GK \big[ \Delta y \big](t,x;\xi) 
\quad\text{with gPC modes}\quad
\widehat{\Delta y}(t,x)
=
\begin{pmatrix}
T_{1,1}(x) \indikator & T_{1,2}(x) \indikator \\
T_{2,1}(x) \indikator & T_{2,2}(x) \indikator
\end{pmatrix}
\Rd(t,x),
$$
where~$\indikator\in\mathbb{R}^{|\K|\times |\K| }$ denotes the identity matrix. 
Initial values are deterministic and stated directly in Riemann coordinates as 
$\R^\pm_0(x)= \cos(2\pi x)
$. We choose the parameters~$\kappa_0=\kappa_1=0.9$ in the feedback law~\eqref{feedbacklaw}. The parameters in the weights~\eqref{Weights} 
are~$\mub = 0.25$ and $h_k^\pm=1$. 
We assume  the relation~$
\DerivStress
= 0.02 \sigma^*
$, 
$E = 100$ 
and as discretization we use~$\Delta x = 2^{-8}$, $L=1$, $\textup{CFL} = 0.99$. The gPC  basis is~\eqref{BasisSparse} with Hermite polynomials and truncation $\NKL=K=4$.

\subsubsection{Stabilizing and destabilizing effects of the control and the source term}

The upper panels of Figure~\ref{Instable} are devoted to an unstable test problem. A simulation is shown, where  
the destabilizing effect of the source term is so strong that the applied feedback control is not sufficient to damp the introduced uncertainties. 
The left panel shows the mean of deviations~$\mathbb{E}\big[\sigma(t,x;\xi)-\sigma^*(x)\big]$  
from the desired stress~$\sigma^*=110$ (MPa), which increases in time. 
Likewise, the right panel states that the variance $\mathbb{V}\big[\sigma(t,x;\xi)\big]$ increases exponentially fast. 
At the beginning there is no variance, as initial values are deterministic. 
But the source term introduces uncertainties over time, which are not sufficiently reduced by the boundary control. 
Still, the statement of Theorem~\ref{TheoremStability}  holds. However, the guaranteed rate~$\mu \in \mathbb{R}$ is~\emph{negative} and hence, ensures no decay. 
Furthermore, we note that there are only in this test problem no deviations in the middle of the spatial domain due to the symmetric perturbation.

\begin{figure}[htp]
	\begin{center}
\textbf{\large unstable test problem}
		\scalebox{1}{\includegraphics[width=\linewidth]{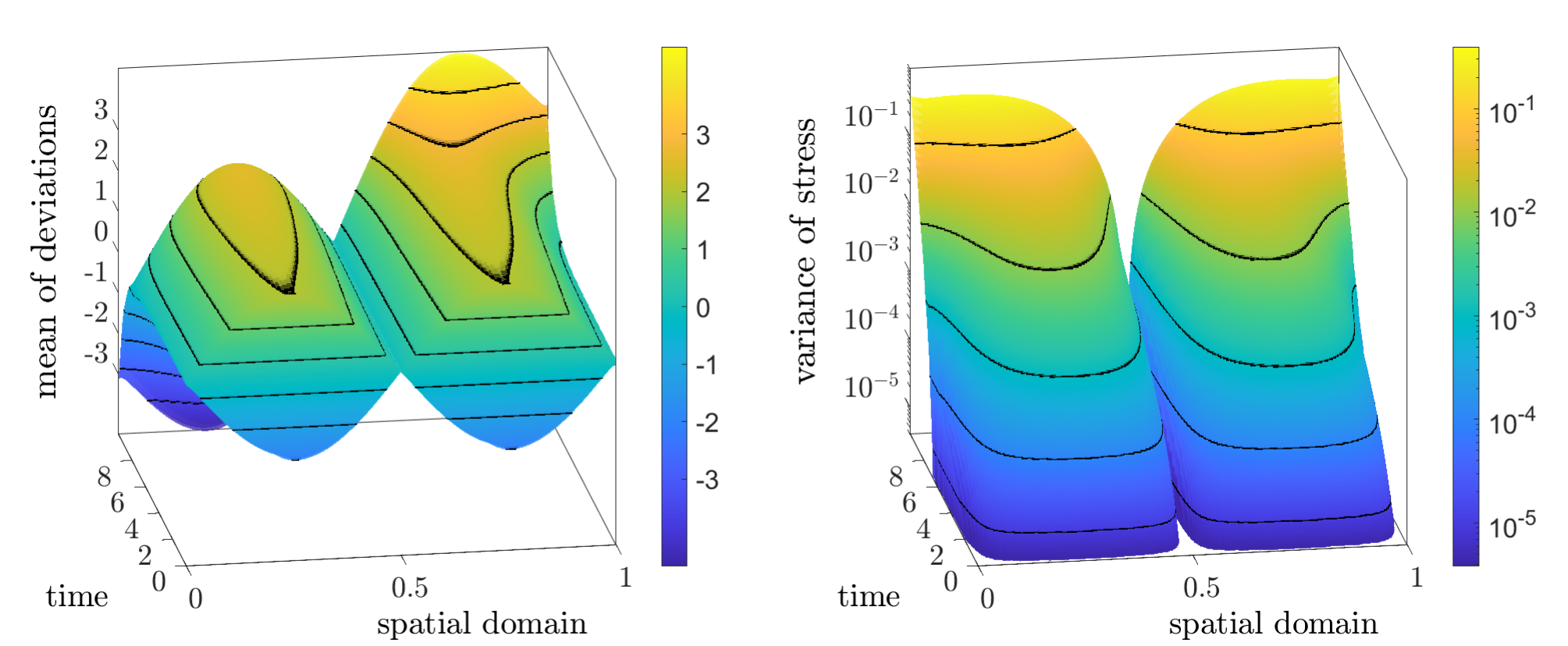}}
\textbf{\large stable test problem}
		\scalebox{1}{\includegraphics[width=\linewidth]{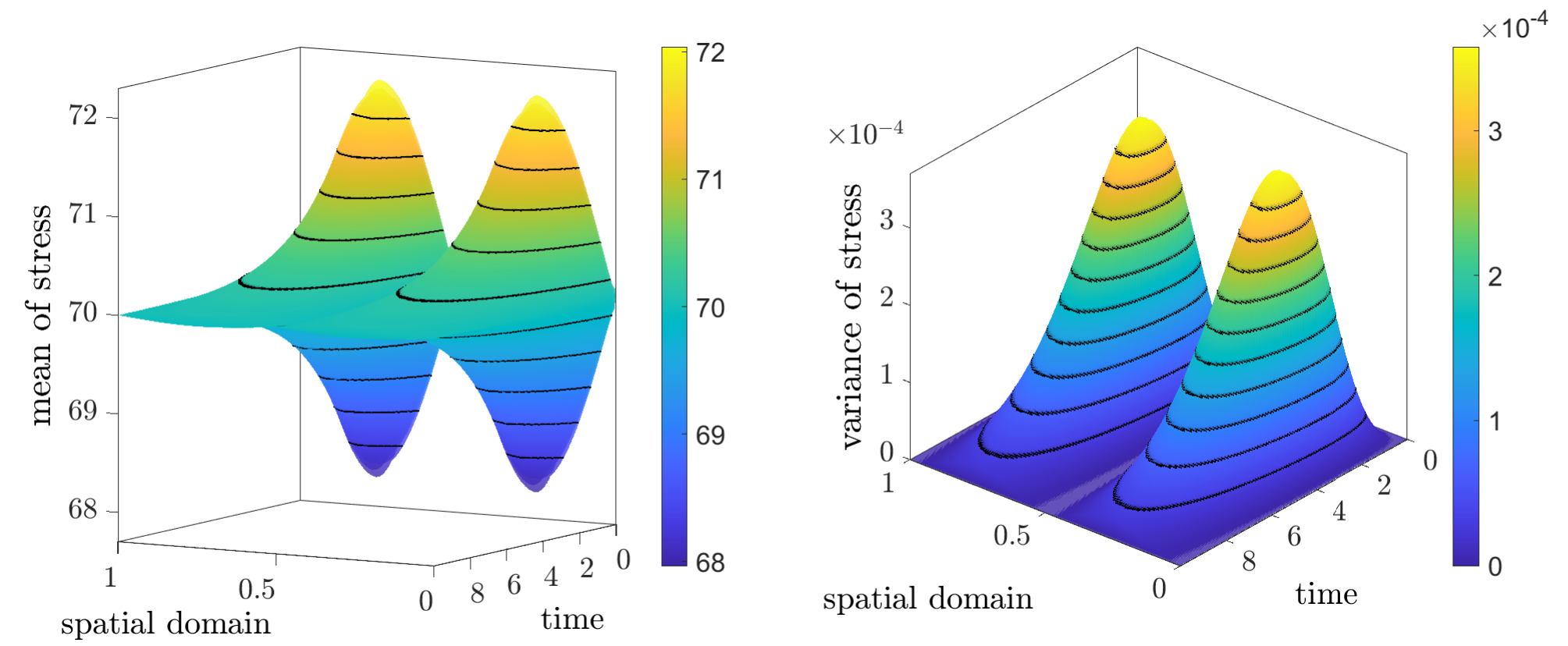}}
	\end{center}
	\caption{Unstable test problem (top): Mean of deviations~$\mathbb{E}\big[\sigma(t,x;\xi)-\sigma^*(x)\big]$  (left) and variance~$\mathbb{V}\big[\sigma(t,x;\xi)\big]$ (right) of uncertain stress~$\sigma$ with desired state~$\sigma^*=110$ (MPa). 	
	Stable test problem (bottom): Mean $\mathbb{E}\big[\sigma(t,x;\xi)\big]$ (left)  and variance~$\mathbb{V}\big[\sigma(t,x;\xi)\big]$ (right) of  stress~$\sigma$ with desired state~$\sigma^*=70$ (MPa).}	
	\label{Instable}
\end{figure}

The lower panels of Figure~\ref{Instable} show a simulation with desired stress~$\sigma^*=70$  (MPa), where destabilizing effects of the source term are smaller. 
Then, the applied feedback control is sufficient to make deviations decay over time and Theorem~\ref{TheoremStability}  ensures a~\emph{positive} decay rate. 
 The left panel shows a convergence of the mean to the desired stress~$\sigma^*=70$  (MPa). At the beginning there is again no variance.  Then,  the variance increases, when the source term introduces uncertainties. However, the feedback law makes also the variance vanish over time. Likewise, there are no deviations in the middle of the spatial domain due to symmetric perturbations.

Figure~\ref{FigLyapunov} illustrates that the derived boundary control~\eqref{feedbacklaw} allows to make the mean squared error~\eqref{MSE} decays exponentially fast. 
A normalized Lyapunov function~$\bar{\L}(t)$, 
which is used in the proof of Theorem~\ref{TheoremStability} 
to make the $L^2$-norm~$\big\lVert \Rd(t,\cdot)\big\rVert_{L^2}$ and hence the mean squared error~\eqref{MSE} decay exponentially fast over time,  
 is plotted in blue. 
The left panel shows a simulation for a stress-strain relationship described by a Bergstr\"om-model with state~$\sigma^*=70$ (MPa)  and the right panel for a DRX-model with state~$\sigma^*=50$, respectively. As illustrated in Figure~\ref{measurements}, the considered stress~$\sigma^*$ is larger in the first case, which leads to a larger destabilizing effect of the source term. Hence, the feedback control steers the system faster to the desired state, when the DRX-model is used.

\begin{figure}[h]
	\begin{center}
		\scalebox{0.8}{\includegraphics[width=\linewidth]{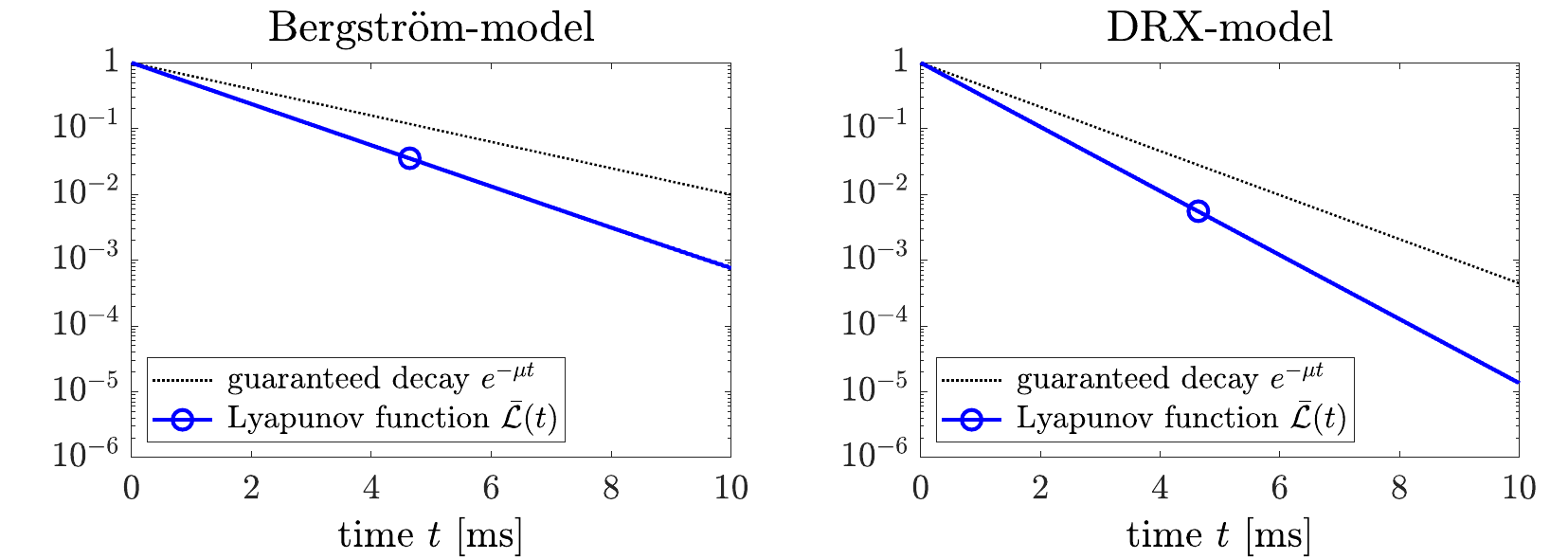}}
	\end{center}
	\caption{Normalized Lyapunov functions are plotted in blue for the feedback control~\eqref{feedbacklaw} with parameters~${\kappa_0=\kappa_1=0.9}$. 
		Theorem~\ref{TheoremStability} guarantees the bound~$\bar{\L}(t)\leq e^{-\mu t}$ (black, dotted). }
	\label{FigLyapunov}
\end{figure}

\subsubsection{Numerical investigation of truncations in the polynomial chaos and Karhunen-Lo\`{e}ve expansion}
Finally, we address the question of optimal truncations in the polynomial chaos and Karhunen-Lo\`{e}ve expansion.  Figure~\ref{FigureTRUNCATION} shows 
for the choices~$K\in\big\{1,\ldots,5\big\}$ and~$M\in\big\{1,\ldots,6\big\}$ the observed decay rate
$$
\mu_{\textup{obs}}
\coloneqq 
-
\frac{1}{t_{\textup{end}}}
\ln\bigg( 
\frac{
\L(t_{\textup{end}})
}{\L(0)} \bigg)
$$
at time~$t_{\textup{end}}=10$. 
The left panel considers the stable case in Figure~\ref{Instable} and the right panel the unstable case. 
Hence, the decay rates are positive in the stable case and negative for the unstable test problem. 
		The dashed, horizontal lines highlight the decay rates for the highest trunction~$K=5$. We observe that the decay rates depend more on the number of random variables than on the degree of ansatz polynomials. Indeed, there are no changes in the rates observable for truncations~$K\geq 3$. An intuitive explanation is given by the relation~\eqref{KLvsPC} in Remark~1, which states that the Karhunen-Lo\`{e}ve expansion is a special case of a gPC expansion with normally distributed random variables and truncation~$K=1$. 
		Although the input uncertainty, given by the linearized plastic part~\eqref{linearizedplasticpart}, requires the gPC truncation~$K=1$, 
		higher truncations are needed to expand the solution to the boundary value problem itself, since it depends on the stochastic input in a nonlinear way. Here,  small truncations with $K=2$ lead to accurate results. 
		However, larger degrees in the ansatz functions might be necessary, when the distribution is far away from being Gaussian.

The choices~$K=5$ and $M=6$ lead to a  set~$\eqref{BasisSparse}$ with 
 462 
elements. 
Hence, the computational results in Figure~\ref{FigureTRUNCATION} also show that the discretization in Riemann coordinates, described in Section~\ref{SectionNumerics}, is able to handle large systems with multiple sources of uncertainties. 
However, computational cost heavily depend on the truncations $(K,M)$. For the considered test problems the truncation~$K=2$ seems optimal. 
For instance,  the choice~$(K,M)=(2,29)$,  when the set~$\eqref{BasisSparse}$ has
465 
elements. 
 leads to a similar system size and hence to a similar computational complexity. 
Therefore, the investigation of an optimal truncation has here the benefit of considering 23 random variables more with the same complexity. 
This motivates trustful indicators that are, however,  subject of further research. 


\begin{figure}[h]
	\begin{center}
		\scalebox{1}{\includegraphics[width=\linewidth]{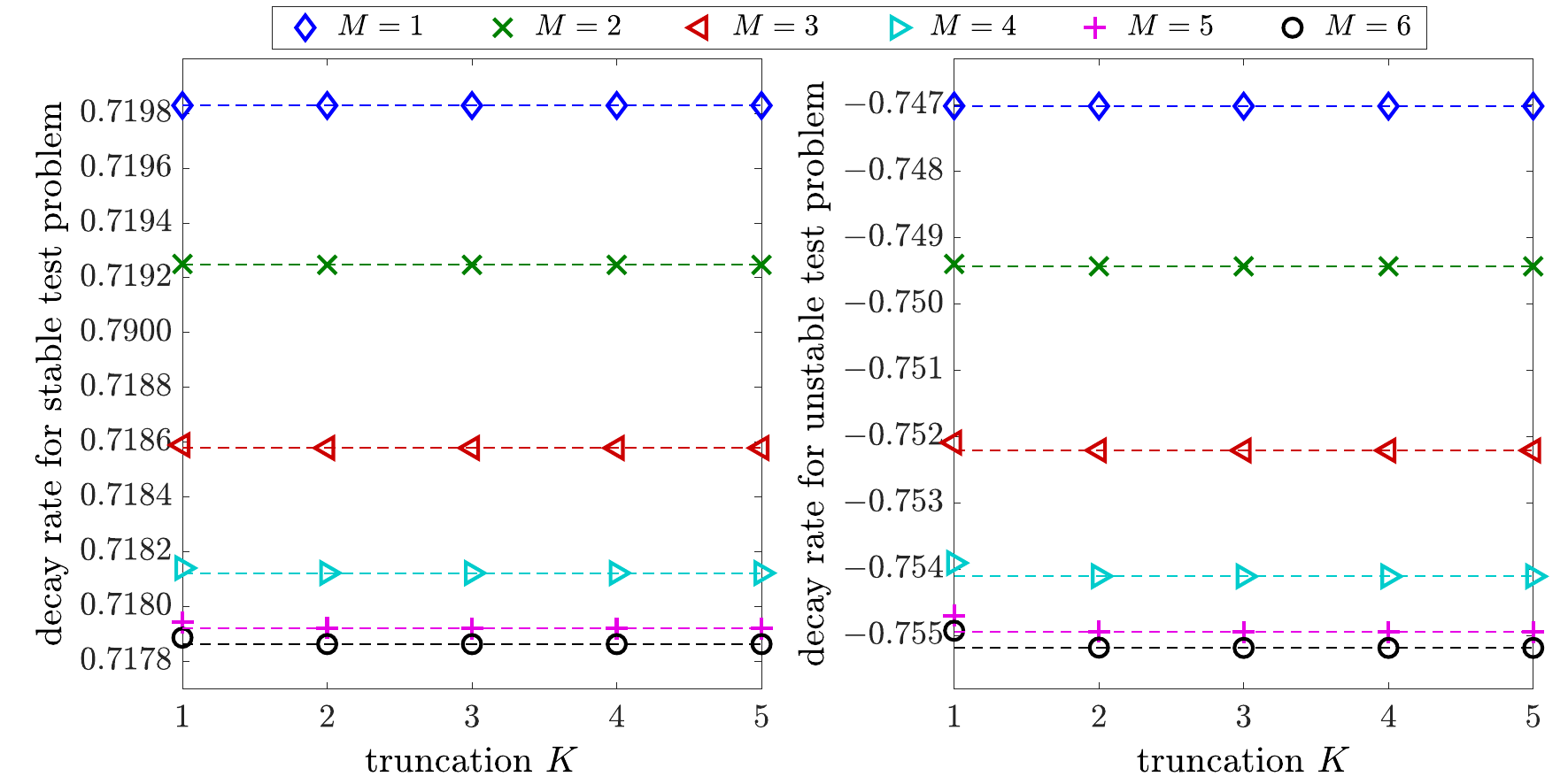}}
	\end{center}
	\caption{Comparison of truncations~$K$ in the polynomial chaos expansion and number~$M$ of random variables in the Karhunen-Lo\`{e}ve expansion for the stable test problem (left) with desired stress~$\sigma^*=70$  (MPa) and the unstable case with stress~$\sigma^*=110$.}	
	\label{FigureTRUNCATION}
\end{figure}

\section{Summary}
We have dealt with the boundary control of random  hyperbolic balance laws to damp deviations of desired states exponentially fast over time. 
Stochastic influences have been introduced as a series of orthogonal functions, called polynomial chaos expansions. Then, a stochastic Galerkin formulation has been derived that  reformulates the underlying random system as a sequence of deterministic problems. 
We have established for linearized  balance laws that  this series comes along with countably  many characteristic curves. 
This allows to define a Lyapunov function that yields an upper bound on mean squared deviations of a desired state. 
A modification of the Lyapunov stability analysis~\cite[Sec.~5]{O1} based on dissipative boundary conditions~\cite[Sec.~2]{L7} makes the Lyapunov function decay  exponentially fast, which in turn results in a diminishing mean squared error.   
Furthermore,  extensions and limitations to nonlinear conservation laws have been discussed. 

Theoretical results have been illustrated by means of a viscoplastic material. 
The deforming process is partially described by a stress-strain relationship under  uncertainties that are modelled by Gaussian random fields. To account for these random fields, we have described in detail the relationship between Karhunen-Lo\`{e}ve and polynomial chaos expansions.

\bigskip
\acknowledgements
The authors thank the Deutsche Forschungsgemeinschaft (DFG,
German Research Foundation) for the financial support through
projects BA4253/11-1 and HE5386/19-1 of the priority program
2183 ``Property-Controlled Forming Processes''. 
Furthermore, this work is supported by the PRIME programme of the German Academic Exchange Service (DAAD) and we would like to offer special thanks to Simone G\"ottlich.

\bibliographystyle{unsrt}

\bibliography{References}
\end{document}